%% file: g2a-v5-arxiv.tex
\documentclass{amsart}
\usepackage{amsthm, amssymb, amsmath, a4wide}

\usepackage{enumitem}
\usepackage{varioref}
\usepackage[bookmarks=false]{hyperref}
\usepackage%[capitalize]
{cleveref}
\crefname{thm}{theorem}{theorems}
\crefname{prop}{proposition}{propositions}

\usepackage{xcolor}
\hypersetup{
    colorlinks,
    linkcolor={red!50!black},
    citecolor={blue!50!black},
    urlcolor={blue!80!black}
}
\usepackage{dsfont}
\newcommand{\id}{\mathds{1}}

\usepackage{subcaption}
\usepackage{tikz}
\usepackage{pgf}
\usetikzlibrary{decorations.pathreplacing,calc}
\usetikzlibrary{calc}
\usepackage[all,cmtip]{xy}

\usepackage{multirow}
\usepackage{tabulary}
\usepackage{booktabs} %for toprule, midrule, bottomrule
\usepackage{bigstrut} %spacing in table below or above \hline

\input{commands-2012}

\author{Pinaki Mondal}
\title{Normal equivariant compactifications of $\mathbb{G}^2_a$ with Picard number one}

\DeclareMathOperator{\aut}{Aut}

\renewcommand{\torus}{\cc^*}
\newcommand{\gtwoa}{\GG^2_a}
\newcommand{\kbarX}{k_{\bar X}}

\newcommand{\baromegastar}{\bar \omega^*}
\newcommand{\xomegatheta}{\bar X_{\vec \omega,\vec \theta}}
\newcommand{\xomegathetaprime}{\bar X_{\vec \omega,\vec \theta'}}
\newcommand{\momega}{m_{\vec \omega}}

\newlist{prooflist}{enumerate}{3}
\setlist[prooflist,1]{label=(\roman*)}
\setlist[prooflist,2]{label=(\arabic*), ref=(\roman{prooflisti}.\arabic*)}
\setlist[prooflist,3]{label=(\alph*),  ref=(\roman{prooflisti}.\arabic{prooflistii}.\alph*)}

\newlist{defnlist}{enumerate}{3}
\setlist[defnlist,1]{label=(\alph*)}
\setlist[defnlist,2]{label=(\arabic*), ref=(\alph{defnlisti}.\arabic*)}
\setlist[defnlist,3]{label=(\roman*), ref=(\alph{defnlisti}.\arabic{defnlistii}.\roman*)}

\begin{document}

\begin{abstract} 
We classify all normal $\GG^2_a$-surfaces with Picard number one, and characterize which of these surfaces have at worst log canonical, and which have at worst log terminal singularities, answering a question of Hassett and Tschinkel \cite{hassett-tschinkel}. We also find all $\gtwoa$-structures on these surfaces and show that these surfaces and their minimal desingularizations have the same $\gtwoa$-structures (modulo equivalence of $\gtwoa$-actions). In particular, we show that some of these surfaces admit one dimensional moduli of $\gtwoa$-structures, answering another question of Hassett and Tschinkel \cite{hassett-tschinkel}. %Our main tool is the classification of 
\end{abstract}

\maketitle

\section{Introduction} \label{sec-intro}

Hassett and Tchinkel \cite{hassett-tschinkel} started the study of {\em $\GG^n_a$-varieties}; these are  equivariant compactifications of $\GG^n_a$, i.e.\ $\cc^n$ with the additive group structure. They classified $\GG^n_a$-structures on projective spaces and Hirzebruch surfaces, and showed that in dimension $\geq 6$ projective spaces admit moduli of $\GG^n_a$-structures. In particular, they asked the following questions regarding the $n = 2$ case:

\begin{problem}[{\cite[Section 5.2] {hassett-tschinkel}}]
\mbox{}
\begin{enumerate}
\item \label{moduli-question} Can the $\gtwoa$-structures on a given (smooth) surface have moduli?
\item \label{log-terminal-question} Classify $\gtwoa$-structures on projective surfaces with log terminal singularities and Picard number one. 
\end{enumerate}
\end{problem}

Motivated by these questions, in this article we undertake a study of normal $\gtwoa$-surfaces with Picard rank one. In particular, we answer both these questions. \\

Indeed, every $\gtwoa$-surface of Picard rank one is trivially a {\em primitive compactification} of $\cc^2$, i.e.\ a compact complex analytic surface containing $\cc^2$ such that the curve at infinity is irreducible. In \cite{sub2-2} we gave an explicit description of automorphisms of normal primitive compactifications of $\cc^2$. Using that description, in this article we classify all normal surfaces with Picard rank one which have $\gtwoa$-structures (\cref{g2a-thm}). Moreover, we give an explicit description of all $\gtwoa$-structures on these surfaces (\cref{g2a-thm}) and of the space of $\gtwoa$-structures modulo equivalence (\cref{structure}). In particular, it turns out that some of these spaces admit one dimensional moduli of $\gtwoa$-structures. On the other hand, we show that normal primitive compactifications of $\cc^2$ have the special property that all of their automorphisms lift to automorphisms of their minimal desingularizations (\cref{min-aut}), which implies that the spaces of $\gtwoa$-structures modulo equivalence on normal $\gtwoa$-surfaces of Picard rank one and on their minimal desingularizations are isomorphic (\cref{isomorphic-structure}). In particular, it follows that some of these minimal desingularizations also admit one dimensional moduli of $\gtwoa$-structures, thereby answering question \eqref{moduli-question}. \\

In \cite{sub2-1} we gave an explicit description of minimal desingularizations of normal primitive compactifications of $\cc^2$. Combining this with Kawamata's \cite{kawamata} classification of log canonical surface singularities (we follow the description of Alexeev \cite{alexeev}) and our classification of $\gtwoa$-structures on normal surfaces of Picard rank one (\cref{g2a-thm,structure}), we immediately obtain a classification of $\gtwoa$-structures on projective surfaces with log terminal or log canonical singularities and Picard number one (\cref{tc-thm}), which answers question \eqref{log-terminal-question}. \\

Some (more precisely, four, up to isomorphism,) of the $\gtwoa$-surfaces of Picard rank one are also {\em singular del Pezzo surfaces} in the sense of Derenthal and Loughran \cite{derenthal-loughran} corresponding to dual graphs of type $A_1, A_2 + A_1, A_4$ and $D_5$ (\cref{del-Pezzo-cor}). In particular, the first two are respectively weighted projective spaces $\pp^2(1,1,2)$ and $\pp^2(1,3,2)$, and have precisely two $\gtwoa$-structures modulo equivalence. The third one admits a one dimensional moduli of $\gtwoa$-structures (modulo equivalence) - it is described in \cref{simple-section}. The other one admits only one $\gtwoa$-structure modulo equivalence. 

\section{A simple non-singular surface with one dimensional moduli of $\gtwoa$-structures.} \label{simple-section}
Let $\bar X:= \pp^2$ and $L$ be a line on $\bar X$. Blow up a point $P$ on $L$, then blow up the point where the strict transform of $L$ intersects the exceptional curve, and then blow up again the point of intersection of the strict transform of $L$ and the new exceptional curve. Finally blow up a point on the newest exceptional curve which is not on the strict transform of either $L$ or any of the older exceptional curves. Let $\bar X'$ be the resulting surface. Identifying $X := \bar X\setminus L$ with $\cc^2$, we see that $\bar X'$ is a non-singular compactification of $\cc^2$ and the `weighted dual graph' of the curve at infinity on $\bar X'$ is as in \cref{fig:simple}. 

\begin{figure}[htp]
\begin{center}
\begin{tikzpicture}
 	\pgfmathsetmacro\edge{1.5}
  	\pgfmathsetmacro\dashededge{1.75}	
 	\pgfmathsetmacro\vedge{0.75}
 	\pgfmathsetmacro\dashedvedge{1}
	\pgfmathsetmacro\vr{0.1}
 	
 	\draw (0,0) -- (2*\edge,0);
 	\fill[black] (0, 0) circle (\vr);
 	\draw (0,0)  node [below] {$L$};
 	\draw (0,0)  node [above] {$-2$};
 	
 	\fill[black] (\edge, 0) circle (\vr);
 	\draw (\edge,0)  node [below right] {$E_3$};
 	\draw (\edge,0)  node [above] {$-2$};
 	
 	\fill[black] (2*\edge, 0) circle (\vr);
 	\draw (2*\edge,0)  node [below] {$E_4$};
 	\draw (2*\edge,0)  node [above] {$-1$};
 	
 	\draw (\edge,0) -- (\edge,-2*\vedge);
	\fill[black] (\edge, -\vedge) circle (\vr);
 	\draw (\edge,-\vedge)  node [right] {$E_2$}; 	
 	\draw (\edge,-\vedge)  node [left] {$-2$};

	\fill[black] (\edge, -2*\vedge) circle (\vr);
 	\draw (\edge,-2*\vedge)  node [right] {$E_1$}; 	
 	\draw (\edge,-2*\vedge)  node [left] {$-2$};
 	 	 	
\end{tikzpicture}
\caption{Weighted dual graph of the curve at infinity on $\bar X'$}\label{fig:simple}
\end{center}
\end{figure}
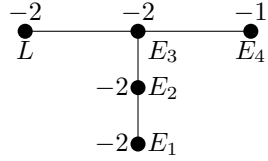
Let $\bar X''$ be the surface formed by contracting (the strict transforms of) $L$, $E_1$, $E_2$ and $E_3$. In the notation introduced in \cref{presection}, $\bar X''$ is the normal primitive compactification of $\cc^2$ corresponding to {\em key sequence} $\vec \omega := (3, 2, 5)$, and $\bar X'$ is the minimal desingularization of $\bar X''$. Choose homogeneous coordinates $[u:v:w]$ on $X$ such that $L = \{w = 0\}$ and $P$ has coordinates $[1:0:0]$. Then $(x,y) := (u/w, v/w)$ are coordinates on $X$. It follows from \cref{structure,isomorphic-structure} that the moduli (up to equivalence) of $\gtwoa$-structures on $\bar X'$ and $\bar X''$ consists of $\gtwoa$-actions $\tau_\lambda$, $\lambda \in \cc$ defined as follows:
\begin{align}
(t_1,t_2) \cdot_{\tau_\lambda} (x,y) 
	= \left(x + \lambda\left( \frac{(t_1)^2}{2} + t_1y\right)+ t_2 , y + t_1 \right) \label{tau-simple}
\end{align}
\section{Preliminaries on normal primitive compactifications of $\cc^2$} \label{presection}

A {\em primitive compactification} of $\cc^2$ is an analytic surface containing $\cc^2$ such that the curve at infinity is irreducible. In this section we recall some properties of normal primitive compactifications of $\cc^2$ from \cite{sub2-2}.

\begin{defn}[Key sequences]
%; cf.\ the definition of `characteristic $\delta$-sequences' in {\cite[p.\ 1116]{sathaye}}]  
\label{key-seqn}
A sequence $\vec\omega := (\omega_0, \ldots, \omega_{n+1})$, $n \in \zz_{\geq 0}$, of integers is called a {\em key sequence} if it has the following properties: 
\begin{enumerate}
\item \label{positive-1-property} $\omega_0 \geq 1$.
\item \label{gcd-1-property} Let $e_k := \gcd(|\omega_0|, \ldots, |\omega_k|)$, $0 \leq k \leq n+1$ and $\alpha_k := e_{k-1}/e_k$, $1 \leq k \leq n+1$. Then $e_{n+1} = 1$, and
\item \label{smaller-property}  $\omega_{k+1} < \alpha_k\omega_k$, $1 \leq k \leq n$.
%\item $d_0 > d_1 > \cdots > d_l$ (or equivalently, $p_k \geq 2$, $1 \leq k \leq l$).  
\end{enumerate} 
Moreover, $\vec \omega$ is called {\em primitive} if $\omega_{n+1} > 0$ (or equivalently, $\omega_k > 0$ for all $k$, $0 \leq k \leq n+1$), and it is called {\em algebraic} if 
\begin{enumerate}
\setcounter{enumi}{3}
\item \label{algebraic-semigroup-property} $\alpha_k\omega_k \in \zz_{\geq 0} \langle \omega_0, \ldots, \omega_{k-1} \rangle$, $1 \leq k \leq n$.
\end{enumerate}
Finally, $\vec \omega$ is called {\em essential} if $\alpha_k \geq 2$ for $1 \leq k \leq n$. Note that
\begin{enumerate}[label=(\alph{enumi})]
\item Given an arbitrary key sequence $(\omega_0, \ldots, \omega_{n+1})$, it has an associated {\em essential subsequence} $(\omega_0, \omega_{i_1}, \ldots, \omega_{i_l}, \omega_{n+1})$ where $\{i_j\}$ is the collection of all $k$, $1 \leq k \leq n$, such that $\alpha_k \geq 2$.
\item If $\vec \omega$ is an algebraic key sequence, then its essential subsequence is also algebraic.
\end{enumerate} 
\end{defn}

\begin{rem}\label{unique-remark}
Let $\vec\omega := (\omega_0, \ldots, \omega_{n+1})$ be a key sequence. It is straightforward to see that property \ref{smaller-property} implies the following: for each $k$, $1 \leq k \leq n$, $\alpha_k\omega_k$ can be {\em uniquely} expressed in the form $\alpha_k\omega_k = \beta_{k,0}\omega_0 + \beta_{k,1} \omega_1 + \cdots + \beta_{k,k-1}\omega_{k-1}$, where $\beta_{k,j}$'s are integers such that $0 \leq \beta_{k,j} < \alpha_j$ for all $j \geq 1$. If $\vec \omega$ is in additional {\em algebraic}, then $\beta_{k,0}$'s of the preceding sentence are {\em non-negative}.
\end{rem}

\begin{defn} \label{omega-theta-defn}
Let $\vec\omega:= (\omega_0, \ldots, \omega_{n+1})$ be a key sequence and $\vec \theta \in \ntorus$. Let $\WP$ be the weighted projective space $\pp^{n+2}(1, \omega_0, \omega_1, \ldots, \omega_{n+1})$ with (weighted) homogeneous coordinates $[w:y_0: \cdots :y_{n+1}]$. We write $\xomegatheta$ for the subvariety of $\WP$ defined by weighted homogeneous polynomials $G_k$, $1 \leq k \leq n$, given by
\begin{gather} \label{projective-definition}
G_k := w^{\alpha_k\omega_k - \omega_{k+1}}y_{k+1} - \left( y_k^{\alpha_k} - \theta_k \prod_{j=0}^{k-1}y_j^{\beta_{k,j}}\right)
\end{gather}
where $\alpha_k$'s and $\beta_{k,j}$'s are as in Remark \ref{unique-remark}. 
\end{defn}

\begin{prop}[{\cite[Proposition 3.4]{sub2-2}}]\label{projective-embedding}
If $\vec \omega$ is primitive and algebraic, then $\xomegatheta$ is a normal primitive algebraic compactification of $\cc^2$. Conversely, every normal primitive algebraic compactification of $\cc^2$ is isomorphic to $\xomegatheta$ for some primitive algebraic key sequence $\vec\omega:= (\omega_0, \ldots, \omega_{n+1})$ and $\vec \theta \in \ntorus$ for some $n \geq 0$. 
\end{prop}

Let $\vec \omega := (\omega_0, \ldots, \omega_{n+1})$ be a key sequence, and $\vec \omega_e := (\omega_{i_0}, \ldots, \omega_{i_{l+1}})$, where $0 = i_0 < i_1 < \cdots < i_{l+1} = n+1$, be the essential subsequence of $\vec \omega$. Define

\begin{align}
\chi_j := \frac{1}{\omega_0}(\omega_{i_j} - \sum_{k = 1}^{j-1}(\alpha_{i_k} - 1)\omega_{i_k}),
							\quad 1 \leq j \leq l+1. \label{beta-i-j}
\end{align}
where $\alpha_1, \ldots, \alpha_{n+1}$ are as in \cref{key-seqn}. Let 
\begin{align}
\scrE_{\vec \omega} 
	:=
	\begin{cases}
	 \{k\frac{\omega_1}{\omega_0}-1: k \in \zz,\  \max\{0, (\chi_{l+1}+1)\frac{\omega_0}{\omega_1}\} < k < \frac{\omega_0}{\omega_1} + 1\} \cup\{0\}
		 & \text{if}\ \omega_1 > 0,\\
	 \{k\frac{\omega_1}{\omega_0}-1: k \in \zz,\  0 < k < (\chi_{l+1}+1)\frac{\omega_0}{\omega_1}\}
	 		 & \text{if}\ \omega_1 < 0. 
	\end{cases} \label{exponents-omega}
\end{align}

Let $\beta \in \qq$. Let
\begin{align}
\hat k(\beta) 
	&:= 
	\begin{cases}
		0 &\text{if}\ \beta \geq \chi_1, \\
		\max\{k: 1 \leq k \leq l+1,\ \beta <  \chi_k\} & \text{otherwise}.
	\end{cases} 
	\label{hat-k-beta} \\
\hat \omega_\beta	
	&:= \omega_0 \beta + \sum_{j=1}^{\hat k(\beta)} (\alpha_{i_j}-1)\omega_{i_j} 
	\label{hat-omega-beta} \\
\hat I_\beta 
	&= 
	\begin{cases}
	\{i: i_{\hat k(\beta)} < i < i_{\hat k(\beta) + 1}\} 
		& \text{if}\ \hat k(\beta) \leq l\\
	\emptyset 
		&\text{if}\ \hat  k(\beta) = l+1.
	\end{cases}	
\end{align}
Note that $\hat \omega_{\chi_j} = \omega_{i_j}$, $1 \leq j \leq l+1$. 

\begin{defn} \label{normal-key-defn}
We say that a key sequence $\vec \omega = (\omega_0, \ldots, \omega_{n+1})$ is in the {\em normal form} if it satisfies one of the following (mutually exclusive) conditions: 
\begin{enumerate}[label = (N\arabic{enumi}), ref=N\arabic{enumi}]
\setcounter{enumi}{-1}

\item \label{trivially-normal} 
\begin{enumerate} 
\item\label{trivially-zero-length} $n = 0$.
\item \label{trivially-omega} $\omega_0 \geq \omega_1$. 
\end{enumerate}
\item \label{generally-normal}
\begin{enumerate}
\item \label{generaly->=1} $n \geq 1$.
\item \label{generally-less} $\omega_0 > \omega_1$.
\item \label{generally-alpha} $\frac{\omega_1}{\omega_0} \not\in \{\frac{1}{k}: k \in \zz,\ k \geq 1\} \cup \{0\}$. 
% either $\omega_1 < 0$ or $\alpha_1 \geq 2$.
\item \label{generally-omega} For each $\beta \in \scrE_{\vec \omega}$, there does not exist $i \in \hat I_\beta$ such that $\omega_i = \hat \omega_\beta$. 
\end{enumerate}
\end{enumerate}
\end{defn}

\begin{prop}[{\cite[Theorems 4.6 and 6.3]{sub2-2}}] \label{projective-normal-embedding}
Let $\bar X$ be a primitive normal algebraic compactification of $\cc^2$.
\begin{enumerate}
\item There exist a unique $n \geq 0$ and a unique  primitive algebraic key sequence $\vec \omega = (\omega_0, \ldots, \omega_{n+1})$ in normal form such that $\bar X \cong \xomegatheta$ for some $\vec \theta \in \ntorus$. 
\item Let $\alpha_i$'s and $\beta_{i,j}$'s be as in \cref{unique-remark}. Moreover, set $\alpha_0 := 1$. Let $\vec \omega_e := (\omega_{i_0}, \ldots, \omega_{i_{l+1}})$ be the essential subsequence of $\vec \omega$. Define $\mu_1, \ldots, \mu_n \in \zz$ as follows: for each $i$, $1 \leq i \leq n$, pick the unique $k$ such that $i_k \leq i < i_{k+1}$, and set
\begin{align*}
\mu_i := \alpha_{i_0} \cdots \alpha_{i_k} - \sum_{j=1}^k \alpha_{i_0} \cdots \alpha_{i_{j-1}} \beta_{i,i_{j}}
\end{align*}
If $\vec \theta' \in \ntorus$ is such that $\bar X \cong \xomegathetaprime$ as well, then there exist $\lambda_1, \lambda_2 \in \torus$ such that
\begin{align*}
\theta'_i = \lambda_1^{-\beta_{i,0}}\lambda_2^{\mu_i}\theta_i,\ i= 1, \ldots, n.
\end{align*} 
\end{enumerate}
\end{prop}

\begin{thm}[{\cite[Theorem 5.2]{sub2-2}}] \label{aut}
Fix a system of coordinates $(x,y)$ on $X := \cc^2$. Let $\vec \omega := (\omega_0, \ldots, \omega_{n+1})$ be a primitive key sequence in normal form, $\vec\theta := (\theta_1, \ldots,\theta_n) \in \ntorus$, and $\bar X := \xomegatheta$ be the corresponding primitive compactification of $X$. Let $\scrG$ be the group of automorphisms of $\bar X$. 
\begin{enumerate}	
\item \label{aut-wt} If  \eqref{trivially-normal} holds, then $\bar X \cong \pp^2(1,\omega_0,\omega_1)$. Fix (weighted) homogeneous coordinates $[z:x:y]$ on $\bar X$. 
\begin{enumerate}
\item \label{aut-wt-0} If $\omega_0 = \omega_1 = 1$, then $\bar X \cong \pp^2$ and $\scrG \cong PGL(3,\cc)$. 
\item \label{aut-wt-1} If $\omega_0 > \omega_1 = 1$, then $\scrG = \{[z:x:y] \mapsto [az+by:cx+f(y,z):dz + ey] : a,b,d,e\in \cc,\ ad -be \neq 0,\ c \in \cc^*,\ f$ is a homogeneous polynomial in $(y,z)$ of degree $\omega_0\}$.
\item \label{aut-wt-2} If $\omega_0 > \omega_1 > 1$, then $\scrG = \{[z:x:y] \mapsto [z:ax+ f(y,z): by + cz^{\omega_1}] : a,b\in \cc^*,\ c\in \cc,\ f$ is a weighted homogeneous polynomial in $(y,z)$ of weighted degree $\omega_0\}$.
\end{enumerate}
\item \label{aut-general} If \eqref{generally-normal} holds, define $\bar\omega_k := \omega_{k}/\alpha_{n+1}$, $0 \leq k \leq n$, and $\baromegastar_k := \alpha_1 \bar \omega_1 + \sum_{j=2}^{k-1}(\alpha_j-1)\bar \omega_j - \bar \omega_k$, $2 \leq k \leq n$, where $\alpha_1, \ldots, \alpha_{n+1}$ are as in \cref{key-seqn}. Set $\baromegastar := \gcd(\baromegastar_2, \ldots, \baromegastar_n)$ (note that $\baromegastar$ is defined only if $n \geq 2$) and 
\begin{align}
k_{\bar X} := -\left(\omega_0 + \omega_{n+1} + 1 - \sum_{k=1}^n (\alpha_k -1)\omega_k\right) \label{kbarX}
\end{align}
where $\alpha_1, \ldots, \alpha_{n+1}$ are as in Definition \ref{key-seqn}. Then $\scrG$ consists of all $F:\bar X \to \bar X$ such that 
\begin{align*}
F|_X& : (x,y) \mapsto (a^{\bar \omega_0} x + f(y), a^{\bar \omega_1} y + c),\ \text{where}\\
a &= \begin{cases}
		\text{an arbitrary element of $\cc^*$} 	& \text{if $n = 1$,}\\	
		\text{an $\baromegastar$-th root of	unity}	& \text{if $n \geq 2$.}
	 \end{cases} \\
c &= \begin{cases}
		0 & \text{if}\ \omega_0 + k_{\bar X} \geq 0, \\
		\text{an arbitrary element in \cc} & \text{otherwise.}
	 \end{cases}
\end{align*}
and $f(y) \in \cc[y]$ is a polynomial such that $\deg(f) \leq -(k_{\bar X} + \omega_1 + 1)/\omega_1$. \qed
\end{enumerate}
\end{thm}

\section{$\GG^2_a$-structures on normal surfaces with Picard rank one} \label{action-section}
Let $G$ be a connected linear algebraic group. A $G$-variety is a variety $Y$ with a fixed left action of $G$ such that the stabilizer of a generic point is trivial and the orbit of a generic point is dense. A $\gtwoa$-surface is a $G$-variety for $G = \gtwoa$. Two left actions $\sigma_1, \sigma_2$ of $G$ on $Y$ are {\em equivalent} if there is a commutative diagram as follows:
%\begin{center}
%\begin{tikzcd}%[swap]
%    G \times Y  \arrow{r}{(\alpha, j)} \arrow{d}{\sigma_1}
%%	\mlnode{Affine algebraic\\ varieties\footnote{what the fuck} $X$}
%		  & G \times Y \arrow{d}{\sigma_2} \\  
%	Y  \arrow{r}{j} 	 
%			& Y
%\end{tikzcd}
\begin{align*}
\xymatrix{
    G \times Y  \ar[r]^{(\alpha, j)} \ar[d]^{\sigma_1}
		  & G \times Y \ar[d]^{\sigma_2} \\  
	Y  \ar[r]^{j} 	 
			& Y
}
\end{align*}
%\end{center}
where $\alpha$ (resp.\ $j$) is an automorphism of $G$ (resp.\ $Y$). \\

\Cref{action-lemma} below studies a class of $\gtwoa$-actions on $\cc^2$ of a very special form. It will be used in the classification of $\gtwoa$-surfaces (\cref{g2a-thm}). 

\begin{lemma} \label{action-lemma}
Let $\Phi$ be the morphism $\gtwoa \times \cc^2 \to \cc^2$ given by 
\begin{align}
(t_1,t_2) \cdot (x,y) = \left(a(t_1, t_2)x + \sum_{i=0}^m b_i(t_1, t_2)y^i, b(t_1, t_2)y + c(t_1, t_2) \right),
\end{align}
where $(t_1,t_2)$ (resp.\ $(x,y)$) are coordinates on $\gtwoa$ (resp.\ $\cc^2$), $m \geq 0$ and $a, b, c, b_1, \ldots, b_m \in \cc[t_1, t_2]$. Then $\Phi$ defines a $\gtwoa$-action on $\cc^2$ iff each of the following conditions holds:
\begin{enumerate}
\item $a(t_1,t_2) = b(t_1,t_2) = 1$ for all $(t_1,t_2) \in \gtwoa$.
\item There are $c_1, c_2 \in \cc$ such that $c(t_1, t_2) = c_1t_1 + c_2t_2$ for all $t_1, t_2 \in \gtwoa$. 
\item If $(c_1, c_2) = (0,0)$, then $b_i$ is linear in $(t_1, t_2)$ for each $i$. 
\item If $(c_1, c_2) \neq (0,0)$, then there exists $\lambda_0, \ldots, \lambda_m, \mu_0 \in \cc$ such that 
\begin{align}
b_i(t_1,t_2) &:= 
	\begin{cases}
	g_0(c_1t_1 + c_2t_2) + \mu_0(\bar c_2t_1 - \bar c_1 t_2) & \text{if}\ i = 0,\\
	g_i(c_1t_1 + c_2t_2)  & \text{if}\ i = 1, \ldots, m,\\
	\end{cases}\ \label{b-i-1}
\end{align}
where for each $ i = 0, \ldots, m$, 
\begin{align}
g_i(r) &:= 	\lambda_i r + \frac{\lambda_{i+1}}{2}\binom{i+1}{1}r^2 + \cdots + \frac{\lambda_{m}}{m-i+1}\binom{m}{m-i}r^{m-i+1} \label{g-i-1}
\end{align}
\end{enumerate}
\end{lemma}

\begin{proof}
See \cref{acsection}. 
\end{proof}

Let $\vec \omega := (\omega_0, \ldots, \omega_{n+1})$ be a primitive key sequence in the normal form, $\vec \theta := (\theta_1, \ldots, \theta_n) \in \ntorus$ and $\bar X := \xomegatheta$ be the corresponding primitive compactification of $\cc^2$. Define $\kbarX$ as in \eqref{kbarX}. 

\begin{thm} \label{g2a-thm}
%Let $\bar X := \xomegatheta$ be a primitive compactification of $\cc^2$ where $\vec \omega := (\omega_0, \ldots, \omega_{n+1})$ is a primitive key sequence in the normal form and $\vec \theta := (\theta_1, \ldots, \theta_n) \in \ntorus$. Define $\kbarX$ as in \eqref{kbarX}. 
\mbox{}
\begin{enumerate}
\item \label{g2a-existence} The following are equivalent:
\begin{enumerate}%[(i)]
\item \label{non-trivial-g2a} $\bar X$ admits the structure of a $\GG^2_a$-surface,
\item \label{q>=0} $\omega_0 + \kbarX < 0$.
\end{enumerate}
\item \label{g2a-equation} Assume $\bar X \not\cong \pp^2$ and that there is a $\GG^2_a$-action $\sigma$ on $\bar X$ which makes it a $\GG^2_a$-surface. Then there is an automorphism $F$ of $\bar X$ such that $F$ such that $X$ is invariant under $\sigma \circ (\id, F)$, where $\id$ is the identity map of $\gtwoa$, and $(\sigma \circ (\id, F))|_X$ is of the form  
\begin{align}
(t_1,t_2) \cdot (x,y) = \left(x + \sum_{i=0}^{\momega} g_i(c_1t_1 + c_2t_2)y^i + \mu(\bar c_2t_1 - \bar c_1t_2) , y + c_1t_1 + c_2t_2 \right), \label{action}
\end{align}
where 
\begin{align}
\momega := \lfloor-(k_{\bar X} + \omega_1 + 1)/\omega_1 \rfloor, \label{m}
\end{align}
$(c_1, c_2) \in \cc^2\setminus\{(0,0)\}$, $\mu \in \cc\setminus \{0\}$, $\bar c_j$ are complex conjugates of $c_j$, and
\begin{align}
g_i(r) &:= 	\lambda_i r + \frac{\lambda_{i+1}}{2}\binom{i+1}{1}r^2 + \cdots + \frac{\lambda_{\momega}}{\momega-i+1}\binom{\momega}{\momega-i}r^{\momega-i+1} \label{g-i}
\end{align}
for some $\lambda_0, \ldots, \lambda_{\momega} \in \cc$.
\item Conversely, for every $g_0, \ldots, g_{\momega}, c_1, c_2$ as above, identity \eqref{action} defines a %$\gtwoa$-action on $\cc^2$ which induces a 
$\gtwoa$-structure on $\bar X$. 
\end{enumerate}
\end{thm}

\begin{proof}
\Cref{aut} implies that given any two copies of $\cc^2$ in $\bar X$, there is an automorphism of $\bar X$ that takes one to the other. Let $F$ be any automorphism of $\bar X$ which maps a $\sigma$-invariant copy of $\cc^2$ to $X$; set $\tau := \sigma \circ (\id, F)$. \Cref{aut} implies that $\tau|_X$ is given by
\begin{align}
(t_1,t_2) \cdot (x,y) = \left(a(t_1, t_2)x + \sum_{i=0}^{\momega} b_i(t_1, t_2)y^i, b(t_1, t_2)y + c(t_1, t_2) \right),
\end{align}
where $\momega$ is from \eqref{m}. Now the result follows from \cref{action-lemma}.
\end{proof}

We continue with the notation from \cref{g2a-thm}. If $\omega_0 + \kbarX < 0$, then \cref{g2a-thm} implies that the following equation defines a $\gtwoa$-structure $\tau_\lambda$ on $\xomegatheta$ for each $\lambda \in \cc$.
\begin{align}
(t_1,t_2) \cdot_{\tau_\lambda} (x,y) 
	= \left(x + \lambda \sum_{i=0}^m \frac{1}{m-i+1} \binom{m}{m-i} t_1^{m-i+1} y^i + t_2 , y + t_1 \right) \label{tau}
\end{align}
where $m := \momega$ from \eqref{m}. Note that $\tau_0$ is simply the action $(t_1,t_2) \cdot (x,y) = (x+t_2, y+t_1)$. 

\begin{thm} \label{structure}
Let $\bar \omega_0,\ldots, \bar \omega_n$ be as in assertion \eqref{aut-general} of \cref{aut}. Assume $\bar X \not\cong \pp^2$ and $\omega_0 + \kbarX < 0$.
\begin{enumerate}
\item If $m = 0$, then every $\gtwoa$-structure on $\bar X$ is equivalent to $\tau_0$.
\item If $n = 0$ (or equivalently, $\bar X$ is isomorphic to a weighted projective variety) and $m > 0$, then up to equivalence there are precisely two $\gtwoa$-structures on $\bar X$, namely $\tau_0$ and $\tau_1$.
\item If $n \geq 1$ and $m > 0$, then
\begin{enumerate}
\item every $\gtwoa$-structure on $\bar X$ is equivalent to $\tau_\lambda$ for some $\lambda \in \cc$;
\item if $n = 1$, then $\tau_\lambda$ is equivalent to $\tau_{\lambda'}$ iff $\lambda' = \zeta^{\bar \omega_0}\lambda$ for some $\bar \omega_1$-th root $\zeta$ of identity;
\item if $n >1$, then $\tau_\lambda$ is equivalent to $\tau_{\lambda'}$ iff $\lambda' = \zeta^{\bar \omega_0}\lambda$ for some $\bar \omega'$-th root $\zeta$ of identity, where $\bar \omega' := \gcd(\bar \omega_1, \ldots, \bar \omega_n)$. 
\end{enumerate}

\end{enumerate}
\end{thm}

\begin{proof}
The $m=0$ case follows immediately from \cref{g2a-thm}. So assume $m \geq 1$. Let $X$ be a copy of $\cc^2$ in $\bar X$ with coordinates $(x,y)$. 
Let $\sigma$ be an arbitrary $\gtwoa$-action on $\bar X$. \Cref{aut} implies that replacing $\sigma$ by a $\gtwoa$-equivalent action if necessary, we may assume that $X$ is invariant under $\sigma$. From \eqref{action} it is straightforward to see that after a change of coordinates on $\gtwoa$ we may assume that the action of $\sigma$ on $X$ has the following form :  
\begin{align}
(t_1,t_2) \cdot_\sigma (x,y) = \left(x + \sum_{i=0}^m g_i(t_1)y^i + t_2 , y + t_1 \right), \label{action-1}
\end{align}
where $g_i$'s are defined as in \eqref{g-i}. We would like to understand when there are automorphisms $F$ of $\bar X$ and $\phi$ of $\gtwoa$ which induce an $\gtwoa$-equivalence of $\sigma$ and $\tau_\lambda$. \Cref{orbit-lemma} below implies that $F(X) = X$. Therefore \cref{aut} implies that $F|_X$ is of the form
\begin{align}
(x,y) \mapsto (ax + h(y), dy + e) \label{F}
\end{align}
for some $h \in \cc[y]$ such that $\deg(h) \leq m$. 
%Note that if $\bar X \cong \pp^2(1,1,q)$ for some $q$, then every automorphism of $\bar X$ satisfies \eqref{F} (). \\
%
The $\gtwoa$-equivalence of $\sigma$ and $\tau_\lambda$ is equivalent to the identity
\begin{align}
F((t_1, t_2) \cdot_{\sigma} (x,y)) &= \phi(t_1, t_2) \cdot_{\tau_\lambda} F(x,y) \label{F-equivalence}
\end{align}
Writing $\phi(t_1, t_2) = (s_1, s_2)$ identity \eqref{F-equivalence} becomes
\begin{align*}
F\left(x + \sum_{i=0}^m g_i(t_1)y^i + t_2 , y + t_1 \right)
	&= (s_1, s_2) \cdot_{\tau_\lambda}  (ax + h(y), dy + e) 
\end{align*}
which is equivalent to
\begin{align}
&\left(a\left(x + \sum_{i=0}^m g_i(t_1)y^i + t_2 \right) + h(y+t_1), dy + dt_1 + e \right) \notag \\
&\qquad	\qquad =  \left(ax + h(y) + \lambda \sum_{i=0}^m \frac{1}{m-i+1} \binom{m}{m-i} s_1^{m-i+1} (dy + e)^i + s_2 , dy + e + s_1 \right) \label{F-equivalence-2}
\end{align}
Comparing the $y$-coordinates of both sides of \eqref{F-equivalence-2} gives that $s_1 = dt_1$. Write $h(y) = \sum_{i=0}^m h_iy^i$, $h_i \in \cc$, $i = 0, \ldots, m$. Then a comparison of the $x$-coordinates of both sides of \eqref{F-equivalence-2} implies that 
\begin{align*}
a\sum_{i=0}^m g_i(t_1)y^i + at_2 
	+ \sum_{i=0}^{m} \sum_{j=0}^{i-1}h_i \binom{i}{j}t_1^{i-j}y^j
	=  \lambda \sum_{i=0}^m \frac{1}{m-i+1} \binom{m}{m-i} t_1^{m-i+1} \sum_{j=0}^i \binom{i}{j} e^{i-j}d^jy^j + s_2
\end{align*}
or equivalently,
\begin{align*}
a g_m(t_1)y^m 
	&+ \sum_{j=0}^{m-1}\left(a g_j(t_1) 
	+ \sum_{i=j+1}^{m} h_i \binom{i}{j}t_1^{i-j} \right) y^j + at_2 \\
	&=  \lambda d^my^m + \lambda \sum_{j=0}^{m-1}  \sum_{i=j}^m  \frac{1}{m-i+1} \binom{m}{m-i} \binom{i}{j} e^{i-j}  t_1^{m-i+1} d^jy^j + s_2
\end{align*}
Write
\begin{align}
g_j(t_1) &= 	\sum_{k= j}^m \frac{\lambda_k}{k-j+1}\binom{k}{k-j}t_1^{k-j+1},
\end{align}
Comparing coefficients of $y^m$ gives
\begin{align}
& a\lambda_m = d^m\lambda \label{lambda_m}
\end{align}
For each $j = 1, \ldots, m-1$, comparing coefficients of $y^j$ gives 
\begin{align}
& \frac{a\lambda_m}{m-j+1}\binom{m}{j}t_1^{m-j+1}  +
\sum_{l= 1}^{m-j} \left(\frac{a\lambda_{l+j-1}}{l}\binom{l+j-1}{j} +  h_{l+j} \binom{l+j}{j} \right) t_1^{l} \\
%&\qquad \qquad	= 
%	     \frac{d^j\lambda}{m-j+1} \binom{m}{j}  t_1^{m-j+1}  + 
%	 \lambda  \sum_{i=j+1}^m  \frac{1}{m-i+1} \binom{m}{m-i} \binom{i}{j} e^{i-j}  t_1^{m-i+1} d^j\\
&\qquad \qquad	= 
	     \frac{d^j\lambda}{m-j+1} \binom{m}{j}  t_1^{m-j+1}  + 
	   \sum_{l=1}^{m-j}  \frac{\lambda d^j}{l} \binom{m}{l-1} \binom{m-l+1}{j} e^{m-l+1-j}  t_1^l \label{y^j} 
\end{align}	 
and similarly, for $j = 0$, we have
\begin{align}
& \frac{a\lambda_m}{m+1}t_1^{m+1}  +
\sum_{l= 1}^m \left( \frac{a\lambda_{l-1}}{l} +  h_l \right) t_1^{l}  + at_2
	= 
	     \frac{\lambda}{m+1}  t_1^{m+1}  + 
	   \sum_{l=1}^m \frac{\lambda}{l} \binom{m}{l-1} e^{m-l+1}  t_1^l + s_2 \label{y^0}
\end{align}	 
Identity \eqref{lambda_m} and a comparison of the coefficients of $t_1^2$ from \eqref{y^j} for $j = m-1$ imply that $d = 1$. But then it is straightforward to check that 
\begin{align}
h_k &= - \frac{a\lambda_{k-1}}{k} + \frac{1}{m+1}\binom{m+1}{k}e^{m+1 - k},\quad k=1, \ldots, m \\
s_2 &= at_2
\end{align}
is a solution to equations \eqref{y^j} and \eqref{y^0}. It follows that $\sigma$ is equivalent to $\tau_\lambda$ if and only if it can be arranged that $a\lambda_m = \lambda$ and $d = 1$. The theorem now follows from \cref{aut}. 
\end{proof}

\begin{defn}
A {\em singular del Pezzo surface} is a singular normal projective surface $Y$ with only {\em $ADE$-singularities}, i.e.\ singularities for which the dual graphs of minimal resolutions are Dynkin diagrams of type $A_k$, $D_l$, or $E_m$ for some $k, l$ or $m$, and each of the irreducible exceptional curve are rational curves with self intersection number $-2$. The {\em type} of a singular point on $Y$ is the type (as a Dynkin diagram) of the weighted dual graph of its minimal resolution.
\end{defn}

\begin{cor} \label{del-Pezzo-cor}
Adopt the notation of \cref{g2a-thm}. 
\begin{enumerate}
\item \label{del-Pezzo-assertion} The following are equivalent:
\begin{enumerate}
\item $\xomegatheta$ is a singular del Pezzo surface which admits a $\gtwoa$-structure.
\item \label{del-Pezzo-omega} $\vec \omega$ is either $(2, 1)$, $(3,2)$, $(3,2,5)$ or $(3,2,4)$. 
\end{enumerate}
\item If $\vec \omega = (2,1)$, then $\xomegatheta \cong \pp^2(1,1,2)$; $\xomegatheta$ has only one singular point and it is of type $A_1$. Up to equivalence $\xomegatheta$ has precisely two $\gtwoa$-structures. 
\item If $\vec \omega = (3,2)$, then $\xomegatheta \cong \pp^2(1,3,2)$; $\xomegatheta$ has two singular points - one of type $A_2$ and the other of type $A_1$. Up to equivalence $\xomegatheta$ has precisely two $\gtwoa$-structures. 
\item If $\vec \omega = (3,2,5)$, then $\xomegatheta$ is isomorphic to the hypersurface in $\pp^2(1,3,2,5)$ (with weighted homogeneous coordinates $[w:x:y:z]$) defined by $wz - (y^3 + x^2) = 0$; $\xomegatheta$ has only one singular point and it is of type $A_4$. Up to equivalence $\xomegatheta$ has a one dimensional moduli of $\gtwoa$-structures. 
\item If $\vec \omega = (3,2,4)$, then $\xomegatheta$ is isomorphic to the hypersurface in $\pp^2(1,3,2,4)$ (with weighted homogeneous coordinates $[w:x:y:z]$) defined by $w^2z - (y^3 + x^2) = 0$; $\xomegatheta$ has only one singular point and it is of type $D_5$. Up to equivalence $\xomegatheta$ has a one dimensional moduli of $\gtwoa$-structures. 
\end{enumerate} 
\end{cor}

\begin{proof}
\cite[Corollary 7.9]{sub2-2} implies that $\xomegatheta$ is a singular del Pezzo surface iff $\vec \omega = (2,1)$, $(3,2)$, $(3,2,5,1)$ or $(3,2,6-r)$, $r = 1, \ldots, 5$. Assertion \eqref{del-Pezzo-assertion} then follows from \cref{g2a-thm} and the observation that $\omega_0 + \kbarX < 0$ for precisely those $\vec \omega$ listed in assertion \eqref{del-Pezzo-omega}. The other assertions follow from \cref{structure}, \cite[corollary 7.9]{sub2-2} and the description in \cite[theorem 4.5]{sub2-1} of weighted dual graphs of minimal resolutions of singularities of primitive compactifications of $\cc^2$. 
\end{proof}

\section{$\gtwoa$-structures on minimal desingularizations of $\gtwoa$-surfaces of Picard rank one}

\begin{lemma} \label{orbit-lemma}
Let $Y$ be an irreducible variety, $G$ be a group, $U$ be an open subset of $Y$, and $\sigma_1$, $\sigma_2$ be two actions of $G$ on $Y$ such that $U$ is an orbit of $G$ under both $\sigma_1$ and $\sigma_2$.  Assume there are automorphisms $\alpha$ of $G$ and $j$ of $Y$ such that $\sigma_2(a(g), j(y)) = j(\sigma_1(g,y))$ for all $g \in G$, $y \in Y$. Then $j(U) = U$.  
\end{lemma}

\begin{proof}
Indeed, since $U$ is open in $Y$, there exist $U \cap j^{-1}(U) \neq \emptyset$. Pick $y \in U \cap j^{-1}(U)$.  Then 
\begin{align*}
U	
	&= \text{$G$-orbit of $j(y)$ under $\sigma_2$} \\
	&= \{\sigma_2(a(g), j(y)): g \in G\}\\
	&= \{ j(\sigma_1(g,y)): g \in G\} \\
	&= j(\text{$G$-orbit of $y$ under $\sigma_1$}) \\
	&= j(U)
\end{align*}
as required.
\end{proof}

Let $\vec \omega$ be a key sequence in normal form, $\bar X := \xomegatheta$ be a $\gtwoa$-surface of Picard rank $1$ containing $X \cong \cc^2$, and $\pi: \bar X^{\min} \to \bar X$ be the minimal desingularization. Let $\sigma_1$ and $\sigma_2$ be actions of $\gtwoa$ on $\bar X$ which fix $X$. Assertion \eqref{lift} of \cref{min-aut} implies that there are group actions $\sigma'_j$, $1 \leq j \leq 2$ of $\gtwoa$ on $\bar X^{\min}$ which are compatible with $\sigma_j$ and $\pi$, i.e.\ $\pi(\sigma'_j(a,x')) = \sigma_j(a,\pi(x'))$ for each $j$ and $a \in \gtwoa$, $x' \in \bar X^{\min}$.

\begin{lemma} \label{going-down-equivalence}
If $\sigma'_1$ is equivalent to $\sigma'_2$, then $\sigma_1$ is equivalent to $\sigma_2$.
\end{lemma}

\begin{proof}
Assume $\sigma'_1$ is equivalent to $\sigma'_2$ via automorphisms $\alpha$ of $\gtwoa$ and $j$ of $\bar X^{\min}$. Assertion \eqref{descend} of \cref{min-aut} implies that $j$ induces an automorphism $\bar j$ of $\bar X$. It is straightforward to see that $\sigma'_1$ is equivalent to $\sigma'_2$ via $(\alpha, \bar j)$.
\end{proof}

\begin{cor} \label{isomorphic-structure}
The space of $\gtwoa$-structures modulo equivalence on $\bar X^{\min}$ is isomorphic to the space of $\gtwoa$-structures modulo equivalence on $\bar X$.
\end{cor}

\begin{proof}
It follows immediately from combining \cref{structure,going-down-equivalence} once we observe that every $\gtwoa$-action on $\bar X^{\min}$ is equivalent to an action which fixes $X$. 
\end{proof}

\begin{cor}
If $\omega_0 + \kbarX < 0$ and $m_{\vec \omega} \geq 1$, then the minimal resolution of singularities of $\bar X$ admits a moduli of $\gtwoa$-structures. \qed 
\end{cor}

\section{Log terminal and log canonical and $\gtwoa$-surfaces with Picard rank one}
In \cref{tc-section} we recall following Alexeev \cite{alexeev} a part of Kawamata's \cite{kawamata} classification of two dimensional log terminal and log canonical singularities in terms of dual graphs of their resolutions of singularities. In \cref{primitive-dual-section} we recall from \cite{sub2-1} the description of dual graphs of resolution of singularities of primitive normal compactifications of $\cc^2$. In \cref{ptc-section} we combine these results to classify log terminal and log canonical primitive normal compactifications of $\cc^2$, and among these characterize those which admit $\gtwoa$-structures. 

\subsection{Two dimensional log terminal and log canonical singularities} \label{tc-section}

\begin{defn} \label{dual-defn}
Let $E_1, \ldots, E_k$ be non-singular curves on a (non-singular) surface such that for each $i \neq j$, either $E_i \cap E_j = \emptyset$, or $E_i$ and $E_j$ intersect transversally at a single point. Then $E = E_1 \cup \cdots \cup E_k$ is called a {\em simple normal crossing curve}. The {\em (weighted) dual graph} of $E$ is a weighted graph with $k$ vertices $V_1, \ldots, V_k$ such that 
\begin{itemize}
\item there is an edge between $V_i$ and $V_j$ iff $E_i \cap E_j \neq \emptyset$,
\item the weight of $V_i$ is the self intersection number of $E_i$.
\end{itemize}
Usually we will abuse the notation, and label $V_i$'s also by $E_i$. If $\pi: Y' \to Y$ is the resolution of singularities of a surface, then the weighted dual graph of $\pi$ is the weighted dual graph of the union of `exceptional curves' (i.e.\ the curves that contract to points under $\pi$) of $\pi$. 
\end{defn}

\begin{defn}
Let $(Y,P)$ be a germ of a normal surface, and $\pi:Y' \to Y$ be a resolution of the singularity of $(Y,P)$ such that the inverse image of $P$ is a normal crossing curve $E$. If $K_Y, K_{Y'}$ are respectively canonical divisors of $Y$, then 
$$K_{Y'} = \pi^*(K_Y) + \sum_j a_j E_j$$
where the sum is over irreducible components $E_j$ of $E$ and $a_j$ are rational numbers. The singularity $(Y,P)$ is called {\em log canonical} (resp.\ {\em log terminal}) if $a_j \geq -1$ (resp.\ $a_j > -1$) for all $j$. 
\end{defn}

\begin{defn}
Given a simple normal crossing curve $E$ on a surface, we denote by $\Delta(E)$ the absolute value of the determinant of the matrix of intersection numbers of components of $E$. If $\Gamma$ is the weighted dual graph of $E$, then we also write $\Delta(\Gamma)$ for $\Delta(E)$. 
\end{defn}

The result below is a special case of Kawamata's classification of log terminal and log canonical singularities. We follow the notation of Alexeev. The dual graphs of possible resolutions are listed in \cref{figure-t,figure-c}. The notation used in these figures is as follows: 
\begin{itemize}
\item each dot denotes a vertex;
\item a number next to a dot represents the weight of the vertex;
\item each empty oval denotes a {\em chain}, i.e.\ a tree such that every vertex has at most two edges; 
\item an oval with a symbol $\Delta$ in the interior denotes a chain $\Gamma$ with $\Delta(\Gamma) = \Delta$. 
\end{itemize}

\begin{thm}[Kawamata \cite{kawamata}, Alexeev \cite{alexeev}] \label{ct-thm}
Let $\pi: Y' \to Y$ be a resolution of singularities of a germ $(Y,P)$ of normal surface. Assume that the exceptional curve $E$ of $\pi$ satisfies the following properties: $E$ is a simple normal crossing curve, each irreducible component of $E$ is a rational curve, and the dual graph of $E$ is a tree.  
\begin{enumerate}
\item The singularity of $(Y,P)$ is log terminal iff the dual graph $\Gamma$ of $E$ is one of the graphs listed in \cref{figure-t}. 
\begin{figure}[htp]
\begin{center}

\begin{subfigure}{0.2\textwidth}
\begin{tikzpicture}
 	\pgfmathsetmacro\edge{1}
 	\pgfmathsetmacro\vedge{.75}
 	\pgfmathsetmacro\dashedvedge{1}
 	\pgfmathsetmacro\ovalx{.6}
 	\pgfmathsetmacro\ovaly{0.25}
	\pgfmathsetmacro\diffx{1.5}
	\pgfmathsetmacro\vx{0.1}
	
 	\draw (0,0) circle [x radius=\ovalx, y radius=\ovaly]; 

\end{tikzpicture}
\end{subfigure}
\begin{subfigure}{0.4\textwidth}
\begin{tikzpicture}
 	\pgfmathsetmacro\edge{1}
 	\pgfmathsetmacro\vedge{.75}
 	\pgfmathsetmacro\dashedvedge{1}
 	\pgfmathsetmacro\ovalx{.6}
 	\pgfmathsetmacro\ovaly{0.25}
	\pgfmathsetmacro\diffx{1.5}
	\pgfmathsetmacro\vx{0.1}
	
	 \draw (0,0) circle [x radius=\ovalx, y radius=\ovaly]; 
	 \draw (0,0 )  node {$\Delta_1$};
	 \draw (\ovalx, 0) -- (\ovalx + 2*\edge, 0);
	 \fill[black] (\ovalx + \edge, 0) circle (\vx);
	 \draw (2*\ovalx + 2*\edge,0) circle [x radius=\ovalx, y radius=\ovaly]; 
	 \draw (2*\ovalx + 2*\edge ,0 )  node {$\Delta_3$};
	 \draw (\ovalx + \edge, 0) -- (\ovalx + \edge, -\vedge);
	 \draw (\ovalx + \edge,-\vedge-\ovalx) circle [x radius=\ovaly, y radius=\ovalx]; 
	 \draw (\ovalx + \edge,-\vedge-\ovalx )  node {$\Delta_2$};
	 
	 \draw (\ovalx + 2*\edge,-\vedge) node [text width= 3.95cm, below right]  {$(\Delta_1, \Delta_2, \Delta_3) \in \{(2,2,n)$, $(2,3,3)$, $(2,3,4)$, $(2,3,5) \}$};
\end{tikzpicture}
\end{subfigure}
\caption{Dual graphs of resolutions of relevant log terminal singularities}
\label{figure-t}
\end{center}
\end{figure}
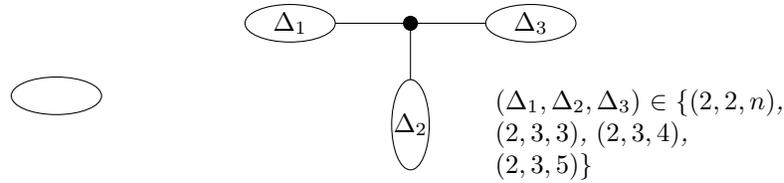
\item The singularity of $(Y,P)$ is log canonical but not log terminal iff the dual graph $\Gamma$ of $E$ is one of the graphs listed in \cref{figure-c}. 

\begin{figure}[htp]
\begin{center}
%arguments: bottom_angle, bottom_xradius, bottom_yradius, direction_angle, length, radius_node, weight, weight_pos
\newcommand{\drawsedge}[8]{

 	\coordinate (P1) at  ($(#1:{#2} and {#3})$);
 	\coordinate (P11) at ($(P1) + (#4:#5)$);
 	\draw (P1) -- (P11);
 	\fill[black] (P11) circle (#6);  
	\draw (P11) node [#8] {#7}; 
}

\begin{subfigure}{0.2\textwidth}
\begin{tikzpicture}
 	\pgfmathsetmacro\sedge{.75}
 	\pgfmathsetmacro\ovalx{.6}
 	\pgfmathsetmacro\ovaly{0.25}
	\pgfmathsetmacro\vx{0.1}
	\pgfmathsetmacro\vanglepos{30}
	\pgfmathsetmacro\vangledir{45}
	
 	\draw (0,0) circle [x radius=\ovalx, y radius=\ovaly]; 
 	
 	\coordinate (P1) at  ($(\vanglepos:0.6 and 0.25)$);
% 	\coordinate (P11) at ($(P1) + (\vangledir:\sedge)$);
% 	\draw (P1) -- (P11);
% 	\fill[black] (P11) circle (\vx);  
%	\draw (P11) node [above] {2}; 
	\drawsedge{\vanglepos}{\ovalx}{\ovaly}{\vangledir}{\sedge}{\vx}{2}{above};
	\drawsedge{-\vanglepos}{\ovalx}{\ovaly}{-\vangledir}{\sedge}{\vx}{2}{below};
%	\drawsedge{$180-\vanglepos$}{\ovalx}{\ovaly}{\vangledir}{\sedge}{\vx}{2};
 	\drawsedge{180-\vanglepos}{\ovalx}{\ovaly}{180-\vangledir}{\sedge}{\vx}{2}{above};
 	\drawsedge{180+\vanglepos}{\ovalx}{\ovaly}{180+\vangledir}{\sedge}{\vx}{2}{below};
\end{tikzpicture}
\end{subfigure}
\begin{subfigure}{0.15\textwidth}
\begin{tikzpicture}
 	\pgfmathsetmacro\sedge{.75}
 	\pgfmathsetmacro\ovalx{.6}
 	\pgfmathsetmacro\ovaly{0.25}
	\pgfmathsetmacro\vx{0.1}
	\pgfmathsetmacro\vanglepos{30}
	\pgfmathsetmacro\vangledir{45}
	
 	\fill[black] (0,0) circle (\vx); 
 	
 	\coordinate (P1) at  ($(\vanglepos:0.6 and 0.25)$);
% 	\coordinate (P11) at ($(P1) + (\vangledir:\sedge)$);
% 	\draw (P1) -- (P11);
% 	\fill[black] (P11) circle (\vx);  
%	\draw (P11) node [above] {2}; 
	\drawsedge{\vanglepos}{\vx}{\vx}{\vangledir}{\sedge}{\vx}{2}{above};
	\drawsedge{-\vanglepos}{\vx}{\vx}{-\vangledir}{\sedge}{\vx}{2}{below};
%	\drawsedge{$180-\vanglepos$}{\ovalx}{\ovaly}{\vangledir}{\sedge}{\vx}{2};
 	\drawsedge{180-\vanglepos}{\vx}{\vx}{180-\vangledir}{\sedge}{\vx}{2}{above};
 	\drawsedge{180+\vanglepos}{\vx}{\vx}{180+\vangledir}{\sedge}{\vx}{2}{below};
\end{tikzpicture}
\end{subfigure}
\begin{subfigure}{0.4\textwidth}
\begin{tikzpicture}
 	\pgfmathsetmacro\edge{1}
 	\pgfmathsetmacro\vedge{.75}
 	\pgfmathsetmacro\dashedvedge{1}
 	\pgfmathsetmacro\ovalx{.6}
 	\pgfmathsetmacro\ovaly{0.25}
	\pgfmathsetmacro\diffx{1.5}
	\pgfmathsetmacro\vx{0.1}
	
	 \draw (0,0) circle [x radius=\ovalx, y radius=\ovaly]; 
	 \draw (0,0 )  node {$\Delta_1$};
	 \draw (\ovalx, 0) -- (\ovalx + 2*\edge, 0);
	 \fill[black] (\ovalx + \edge, 0) circle (\vx);
	 \draw (2*\ovalx + 2*\edge,0) circle [x radius=\ovalx, y radius=\ovaly]; 
	 \draw (2*\ovalx + 2*\edge ,0 )  node {$\Delta_3$};
	 \draw  (\ovalx + \edge, 0) -- (\ovalx + \edge, -\vedge);
	 \draw (\ovalx + \edge,-\vedge-\ovalx) circle [x radius=\ovaly, y radius=\ovalx]; 
	 \draw (\ovalx + \edge,-\vedge-\ovalx )  node {$\Delta_2$};
	 
	 \draw (\ovalx + 2*\edge,-\vedge) node [text width= 3.9cm, below right]  {$(\Delta_1, \Delta_2, \Delta_3) \in \{3,3,3)$, $(2,4,4)$, $(2,3,6) \}$};
\end{tikzpicture}
\end{subfigure}

\caption{Dual graphs of resolutions of relevant log canonical but not log terminal singularities}
\label{figure-c}
\end{center}
\end{figure}
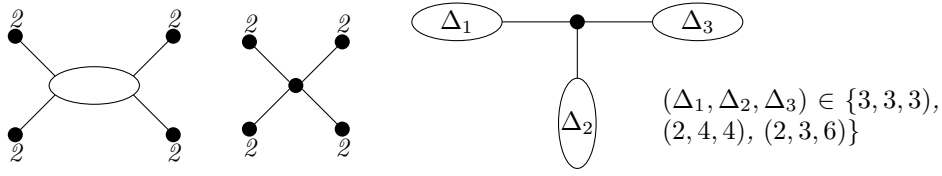

\end{enumerate}
\end{thm}

\subsection{Dual graphs of resolution of singularities of normal primitive compactifications of $\cc^2$} \label{primitive-dual-section}
Let $n \geq 0$, $\vec \omega = (\omega_0, \ldots, \omega_{n+1})$ be a primitive key sequence in normal form, $\vec \theta \in \ntorus$, and $\bar X \cong \xomegatheta$ be the corresponding primitive normal compactification of $\cc^2$. Let $C_\infty$ be the curve at infinity on $\bar X$. It turns out that one can associate a {\em formal descending Puiseux series}, i.e.\ a formal sum 
\begin{align*}
\tilde \phi(x,\xi) = a_1 x^{\beta_1} + \ldots + a_s x^{\beta_s} + \xi x^{\beta_{s+1}}
\end{align*}
where 
\begin{itemize}
\item $s \geq 0$,
\item $a_1, \ldots, a_s \in \cc^*$,
\item $\beta_1 >  \cdots > \beta_{s+1}$ are rational numbers,
\item $\xi$ is an indeterminate,
\end{itemize}
such that for each $f \in \cc[x,y]$, 
\begin{align}
\pole_{C_\infty}(f) = \omega_0 \deg_x\left(f|_{y = \tilde \phi(x,\xi)} \right)
\end{align}
Note that the normality of $\vec \omega$ implies that either $s = 0$, or $\beta_1$ is a positive rational number between $0$ and $1$ such that neither $\beta_1$ nor $1/\beta_1$ is an integer. \\

Let $d_j$ be the lowest common denominator of (reduced forms of) $\beta_1, \ldots, \beta_j$, $1 \leq j \leq s+1$. The sequence of {\em formal characteristic exponents} of $\tilde \phi$ is the sequence $\beta_{1} = \beta_{j_1} >  \cdots > \beta_{j_{l+1}} = \beta_{s+1}$ of exponents which satisfy the following property:
\begin{itemize}
\item $d_j = d_{j_k}$ for each $j = j_k, \ldots, j_{k +1} - 1$,
\item $d_{j_{k+1}} > d_{j_{k}}$ for each $k = 1, \ldots, l-1$. 
\end{itemize}
The {\em formal Newton pairs} of $\tilde \phi$ are $(q'_1, p_1), \ldots, (q'_{l+1}, p_{l+1})$, where 
\begin{align*}
p_k 
	&= 
	\begin{cases}
	d_{j_1} = d_1 & \text{if}\ k = 1,\\
	d_{j_k}/d_{j_{k-1}} &\text{if}\ k = 2, \ldots, l+1.
	\end{cases}\\
q'_k 
	&= 
	\begin{cases}
		d_1\beta_1 & \text{if}\ k = 1,\\
		d_{j_k}(\beta_{j_k} - \beta_{j_{k-1}}) &\text{if}\ k = 2, \ldots, l+1.
	\end{cases}
\end{align*}

The formal Newton pairs are completely determined by (and also completely determine) the {\em essential subsequence} (\cref{key-seqn}) of the key sequence $\vec \omega$. The relation among them is as follows:

\begin{prop}[{\cite[Propositoin A.1]{sub2-2}}] \label{omega-p}
Let $\vec \omega_e = (\omega_{i_0}, \ldots, \omega_{i_{l'+1}})$ be the essential subsequence of $\vec \omega$ and $\alpha_1, \ldots, \alpha_{n+1}$ be as in \cref{key-seqn}. Then $l' = l$, and for each $j$, $1 \leq j \leq l+1$, 
\begin{align}
p_j &= \alpha_{i_j}, \\
\beta_{i_j} &:= \frac{1}{\omega_0}(\omega_{i_j} - \sum_{k = 1}^{j-1}(\alpha_{i_k} - 1)\omega_{i_k}),
							\quad 1 \leq j \leq l+1.
\end{align}
\end{prop}

\newcommand{\drawhblock}[2]{

	\begin{scope}[shift={(#1,0)}]
		\draw (0,0) circle [x radius=\ovalx, y radius=\ovaly]; 
		\draw (0, 0 )  node {#2};
		\draw (\ovalx, 0) -- (\ovalx + \edge, 0);
		\fill[black] (\ovalx + \edge, 0) circle (\vx);
	\end{scope}
}

%arguments: shift_x, q'_k, p_k
\newcommand{\drawblock}[3]{
	
	\drawhblock{#1}{#2}

	\begin{scope}[shift={(#1,0)}]
		\draw (\ovalx + \edge, 0) -- (\ovalx + \edge, -\vedge);
		\draw (\ovalx + \edge,-\vedge-\ovalx) circle [x radius=\ovaly, y radius=\ovalx]; 
		\draw (\ovalx + \edge,-\vedge-\ovalx )  node {#3};
	\end{scope}
}

\begin{thm}[{\cite[Proposition 4.2]{sub2-1}}] \label{primitive-resolution-thm}
\mbox{}
\begin{enumerate}
\item If $p_{l+1} = 1$, then there is a resolution of singularities of $\bar X$ such that the dual graph is of the form displayed in \cref{primitive-resolution}.
\item  If $p_{l+1} > 1$, then there is a resolution of singularities of $\bar X$ such that the dual graph is a disjoint union of a chain $\Gamma$ with $\Delta(\Gamma) = p_{l+1}$ and a graph of the form displayed in \cref{primitive-resolution}.
%\item If $l = 0$ and $\vec \omega = (1,1)$, then $\bar X \cong \pp^2$.
%\item If $\vec \omega \neq (1,1)$ and $p_{l+1} = 1$, then there is a resolution of singularities of $\bar X$ such that the dual graph is of the form displayed in \cref{primitive-resolution}.
%\item  If $\vec \omega \neq (1,1)$ and $p_{l+1} > 1$, then there is a resolution of singularities of $\bar X$ such that the dual graph is a disjoint union of a chain $\Gamma$ with $\Delta(\Gamma) = p_{l+1}$ and a graph of the form displayed in \cref{primitive-resolution}.
%\item If $l = 0$ and $\vec \omega \neq (1,1)$, there is a resolution of singularities of $\bar X$ such that the corresponding dual graph is a disjoint union of one or two (depending on whether $q'_1 = 1$ or $q'_1 > 1$) chains. 
%\item 
%\begin{enumerate}
%\item If $l \geq 1$ and $p_{l+1} = 1$, then there is a resolution of singularities of $\bar X$ such that the dual graph is of the form displayed in \cref{primitive-resolution}.
%\item If $l \geq 1$ and $p_{l+1} > 1$, then there is a resolution of singularities of $\bar X$ such that the dual graph is a disjoint union of a chain and a graph of the form displayed in \cref{primitive-resolution}.
%\end{enumerate}

\begin{figure}[htp]
\begin{center}

%arguments: shift_x, q'_k
%\newcommand{\drawhblock}[2]{
%
%	\begin{scope}[shift={(#1,0)}]
%		\draw (0,0) circle [x radius=\ovalx, y radius=\ovaly]; 
%		\draw (0, 0 )  node {#2};
%		\draw (\ovalx, 0) -- (\ovalx + \edge, 0);
%		\fill[black] (\ovalx + \edge, 0) circle (\vx);
%	\end{scope}
%}
%
%%arguments: shift_x, q'_k, p_k
%\newcommand{\drawblock}[3]{
%	
%	\drawhblock{#1}{#2}
%
%	\begin{scope}[shift={(#1,0)}]
%		\draw (\ovalx + \edge, 0) -- (\ovalx + \edge, -\vedge);
%		\draw (\ovalx + \edge,-\vedge-\ovalx) circle [x radius=\ovaly, y radius=\ovalx]; 
%		\draw (\ovalx + \edge,-\vedge-\ovalx )  node {#3};
%	\end{scope}
%}

\begin{tikzpicture}
	\pgfmathsetmacro\edge{1}
	\pgfmathsetmacro\vedge{.75}
	\pgfmathsetmacro\dashededge{2}
	\pgfmathsetmacro\ovalx{.6}
	\pgfmathsetmacro\ovaly{0.25}
	\pgfmathsetmacro\vx{0.1}
	
	\drawblock{0}{$|q'_1|$}{$p_1$};
	\draw[dashed] (\ovalx + \edge, 0) -- (\ovalx + \edge + \dashededge , 0);
	\drawblock{2*\ovalx + \edge + \dashededge}{$|q'_l|$}{$p_l$};
	\draw (3*\ovalx + 2*\edge + \dashededge, 0) -- (3*\ovalx + 3*\edge + \dashededge , 0);
	\draw (4*\ovalx + 3*\edge + \dashededge,0) circle [x radius=\ovalx, y radius=\ovaly]; 
	\draw (4*\ovalx + 3*\edge + \dashededge,0) node {$|q'_{l+1}|$}; 
\end{tikzpicture}

\caption{Dual graphs in \cref{primitive-resolution-thm}}
\label{primitive-resolution}
\end{center}
\end{figure}
\end{enumerate}

\end{thm}

\subsection{Primitive compactifications with log terminal and log canonical singularities} \label{ptc-section}

In this section we combine the results of preceding two sections to classify primitive normal compactifications of $\cc^2$, including all Picard rank on $\gtwoa$-surfaces, with log terminal and log canonical singularities.\\

We continue to use the notation from \cref{primitive-dual-section}. In particular, $\vec \omega = (\omega_0, \ldots, \omega_{n+1})$ is a primitive key sequence in normal form, $\vec \theta \in \ntorus$, and $\bar X := \xomegatheta$ is the corresponding primitive normal compactification of $\cc^2$. The following are straightforward consequences of the normality of $\vec \omega$ and \cref{omega-p}:
\begin{prooflist}
\item \label{ob1} $\gcd(q'_1, p'_1) = 1$;
\item \label{ob2} if $l \geq 1$, then $p_1 > 2$, $p_1 > q_1$, and neither $q'_1/p_1$ nor $p_1/q'_1$ is an integer.
\end{prooflist}

\begin{thm} \label{tc-thm}
Let $\vec \omega_e$ be the {\em essential subsequence} (\cref{key-seqn}) of $\vec \omega$. 
\begin{enumerate}
\item \label{pt} $\bar X$ has only log terminal singularities iff $\vec \omega_e$ is one of the key sequences in the first column of \cref{t-table}. $\bar X$ is in addition a $\gtwoa$-surface iff it satisfies the conditions from the 4th column of \cref{t-table}. 
\item \label{pc} $\bar X$ has a log canonical singularity which is not log terminal iff $\vec \omega_e$ is the key sequence from \cref{c-table}. $\bar X$ is in addition a $\gtwoa$-surface iff it satisfies the condition from the 4th column of \cref{c-table}. 
\end{enumerate}
\end{thm}

\begin{proof}
The first statements of assertions \eqref{pt} and \eqref{pc} follow from observations \ref{ob1}, \ref{ob2}, \cref{omega-p,ct-thm,primitive-resolution-thm}. For the criteria for having $\gtwoa$-surface structures, we use \cref{g2a-thm}. If $n = 0$, then $\kbarX + \omega_0 = -\omega_1 - 1 < 0$, so that $\bar X$ is a $\gtwoa$-surface; this explains the first two rows of \cref{t-table}. The key sequences in \cref{c-table} and the remaining cases of \cref{t-table} are of the form $(p_1p_2, q_1p_2, q_1p_1p_2 - r)$, with $\alpha_1 = p_1$ and $\alpha_2 = p_2$. It follows that
\begin{align}
k_{\bar X} + \omega_0
	&= - q_1p_1p_2 + r - 1 + (p_1 -1)q_1p_2 
	= r - 1 - q_1p_2
\end{align}
Therefore $k_{\bar X} + \omega_0 < 0$ iff $q_1p_2 \geq r$. This explains the criteria for $\gtwoa$-surface structures in the 4th columns of \cref{t-table,c-table}.
\end{proof}

\newcommand\voval[2]{
\begin{scope}[shift={#1}]
	\draw (0,0) circle [x radius=\ovaly, y radius=\ovalx]; 
	\draw (0,0) node {#2}; 
\end{scope}
}

\newcommand\hoval[2]{
\begin{scope}[shift={#1}]
	\draw (0,0) circle [x radius=\ovalx, y radius=\ovaly]; 
	\draw (0,0) node {#2}; 
\end{scope}
}

\newcommand{\drawthreeblock}[4]{
	
	\drawblock{#1}{#2}{#3}

	\begin{scope}[shift={(#1,0)}]
		\draw (\ovalx + \edge, 0) -- (\ovalx + 2*\edge, 0);
		\draw (2*\ovalx + 2*\edge,0) circle [x radius=\ovalx, y radius=\ovaly]; 
		\draw (2*\ovalx + 2*\edge,0)  node {#4};
	\end{scope}
}

\begin{center}

\small
\begin{table}[htp]
\begin{tabulary}{\textwidth}{|p{3.05cm}|p{2.3cm}|b{4.8cm}|b{2.25cm}|}
%\toprule
\hline
$\vec \omega_e$ & Formal Newton pairs & Dual graph of resolution of singularities & $\gtwoa$-surface iff 
	\bigstrut \\ \hline
$(1,1)$ & $(1,1)$ & - & always 
	\bigstrut \\ \hline
\multirow{2}[4]{2cm}{$(p,q)$ \newline $p > q \geq 1$, $\gcd(p,q) = 1$} 
& \multirow{2}[4]{*}{$(q,p)$} 
& 
\begin{tikzpicture}[scale = 0.75]
 	\pgfmathsetmacro\ovalx{.6}
 	\pgfmathsetmacro\ovaly{0.25}
 	\pgfmathsetmacro\dtext{4.35}
 		
 	\voval{ (0,0)}{$p$};
 	\draw ( \dtext,0) node [right] {if $q = 1$};
\end{tikzpicture}
& 
\multirow{2}[4]{*}{always} 
\bigstrut \\ \cline{3-3}
&
& 
\begin{tikzpicture}[scale = 0.75]
 	\pgfmathsetmacro\ovalx{.6}
 	\pgfmathsetmacro\ovaly{0.25}
 	\pgfmathsetmacro\dtext{4}
 	\pgfmathsetmacro\dU{0.15}
 	\pgfmathsetmacro\wU{0.2} 	

 	\hoval{(0,0)}{$q$};
	\draw (\ovalx+\dU,0) node[text width=\wU] {$\bigcup$}; 
% 	\draw (2*\ovalx+2.5*\dU+\wU,0) circle [x radius=\ovalx, y radius=\ovaly]; 
% 	\draw (2*\ovalx+2.5*\dU+\wU,0) node {$q$};  	

 	\voval{(\ovalx+\ovaly+3*\dU+\wU,0)}{$p$};
 	 	
 	\draw (\dtext,0) node [right] {if $q > 1$};
\end{tikzpicture}
&
\bigstrut \\ \hline
\multirow{2}[4]{3.5cm}{$(p_1p_2, q_1p_2, q_1p_1p_2-1)$ \newline $p_1 > q_1 > 1$, $\gcd(p_1,q_1) = 1$, $p_2 \geq 1$} 
&$(q_1, p_1), (- 1,p_2)$
&
\begin{tikzpicture}[scale = 0.75]
	\pgfmathsetmacro\edge{0.5}
	\pgfmathsetmacro\vedge{.35}
	\pgfmathsetmacro\ovalx{.6}
	\pgfmathsetmacro\ovaly{0.25}
	\pgfmathsetmacro\vx{0.1}
	\pgfmathsetmacro\dtext{4}	
	
	\drawblock{0}{$q_1$}{$p_1$};
	\draw ( \dtext,0) node [right] {if $p_2 = 1$};

\end{tikzpicture}
& 
\multirow{2}[4]{*}{always} 
\bigstrut \\ \cline{3-3}
&
& 
\begin{tikzpicture}[scale = 0.75]
	\pgfmathsetmacro\edge{0.5}
	\pgfmathsetmacro\vedge{.35}
	\pgfmathsetmacro\ovalx{.6}
	\pgfmathsetmacro\ovaly{0.25}
	\pgfmathsetmacro\vx{0.1}
	\pgfmathsetmacro\dtext{4}	
 	\pgfmathsetmacro\dU{0.15}
 	\pgfmathsetmacro\wU{0.2} 	
	
	\drawblock{0}{$q_1$}{$p_1$};
	\draw (\ovalx+\edge+\ovaly+\dU,0) node[text width=\wU] {$\bigcup$}; 
	\voval{(\ovalx+\edge+2*\ovaly+3*\dU+\wU,0)}{$p_2$};
	\draw ( \dtext,0) node [right] {if $p_2>1$};

\end{tikzpicture}
& 
\bigstrut \\ \hline
\multirow{2}[4]{3.5cm}{$(p_1p_2, 2p_2, 2p_1p_2-2)$ \newline $p_1, p_2$ odd, $p_1 > 2$, $p_2 \geq 1$} 
&$(2, p_1), (- 2,p_2)$
&
\begin{tikzpicture}[scale = 0.75]
	\pgfmathsetmacro\edge{0.5}
	\pgfmathsetmacro\vedge{.35}
	\pgfmathsetmacro\ovalx{.6}
	\pgfmathsetmacro\ovaly{0.25}
	\pgfmathsetmacro\vx{0.1}
	\pgfmathsetmacro\dtext{4}	
	
	\drawthreeblock{0}{$2$}{$p_1$}{2};
	\draw ( \dtext,0) node [right] {if $p_2 = 1$};

\end{tikzpicture}
& 
\multirow{2}[4]{*}{always} 
\bigstrut \\ \cline{3-3}
&
& 
\begin{tikzpicture}[scale = 0.75]
	\pgfmathsetmacro\edge{0.5}
	\pgfmathsetmacro\vedge{.35}
	\pgfmathsetmacro\ovalx{.6}
	\pgfmathsetmacro\ovaly{0.25}
	\pgfmathsetmacro\vx{0.1}
	\pgfmathsetmacro\dtext{4}	
 	\pgfmathsetmacro\dU{0.15}
 	\pgfmathsetmacro\wU{0.2} 	
	
	\drawthreeblock{0}{$2$}{$p_1$}{2};
	\draw (3*\ovalx+2*\edge+\dU,0) node[text width=\wU] {$\bigcup$}; 
	\voval{(3*\ovalx+2*\edge+\ovaly+3*\dU+\wU,0)}{$p_2$};
	\draw ( \dtext,0) node [right] {if $p_2 >1$};

\end{tikzpicture}
& 
\bigstrut \\ \hline
\multirow{2}[4]{3.5cm}{$(p_1p_2, 2p_2, 2p_1p_2-r)$ \newline $(p_1, r) \in \{(3,3)$, $(3,4)$, $(3,5)$, $(5,3)\}$,
 $p_2 \geq 1$, $\gcd(p_2,r) = 1$} 
&$(2, p_1), (- r,p_2)$
&
\begin{tikzpicture}[scale = 0.75]
	\pgfmathsetmacro\edge{0.5}
	\pgfmathsetmacro\vedge{.35}
	\pgfmathsetmacro\ovalx{.6}
	\pgfmathsetmacro\ovaly{0.25}
	\pgfmathsetmacro\vx{0.1}
	\pgfmathsetmacro\dtext{4}	
	
	\drawthreeblock{0}{$2$}{$p_1$}{$r$};
	\draw ( \dtext,0) node [right] {if $p_2 = 1$};

\end{tikzpicture}
& 
never
\bigstrut \\ \cline{3-4}
&
& 
\begin{tikzpicture}[scale = 0.75]
	\pgfmathsetmacro\edge{0.5}
	\pgfmathsetmacro\vedge{.35}
	\pgfmathsetmacro\ovalx{.6}
	\pgfmathsetmacro\ovaly{0.25}
	\pgfmathsetmacro\vx{0.1}
	\pgfmathsetmacro\dtext{4}	
 	\pgfmathsetmacro\dU{0.15}
 	\pgfmathsetmacro\wU{0.2} 	
	
	\drawthreeblock{0}{$2$}{$p_1$}{$r$};
	\draw (3*\ovalx+2*\edge+\dU,0) node[text width=\wU] {$\bigcup$}; 
	\voval{(3*\ovalx+2*\edge+\ovaly+3*\dU+\wU,0)}{$p_2$};
	\draw ( \dtext,0) node [right] {if $p_2 >1$};

\end{tikzpicture}
& 
$p_2 \geq 3$ if $(p_1, r) $ $=$ $(3,5)$, $p_2 \geq 2$ otherwise
\bigstrut \\ \hline
\multirow{2}[4]{3.5cm}{$(p_1p_2, 3p_2, 3p_1p_2-2)$ \newline $p_1 = 4,5$,
 $p_2 \geq 1$, $\gcd(p,2) = 1$} 
&$(3, p_1), (- 2,p_2)$
&
\begin{tikzpicture}[scale = 0.75]
	\pgfmathsetmacro\edge{0.5}
	\pgfmathsetmacro\vedge{.35}
	\pgfmathsetmacro\ovalx{.6}
	\pgfmathsetmacro\ovaly{0.25}
	\pgfmathsetmacro\vx{0.1}
	\pgfmathsetmacro\dtext{4}	
	
	\drawthreeblock{0}{$3$}{$p_1$}{$2$};
	\draw ( \dtext,0) node [right] {if $p_2 = 1$};

\end{tikzpicture}
& 
\multirow{2}[4]{*}{always} 
\bigstrut \\ \cline{3-3}
&
& 
\begin{tikzpicture}[scale = 0.75]
	\pgfmathsetmacro\edge{0.5}
	\pgfmathsetmacro\vedge{.35}
	\pgfmathsetmacro\ovalx{.6}
	\pgfmathsetmacro\ovaly{0.25}
	\pgfmathsetmacro\vx{0.1}
	\pgfmathsetmacro\dtext{4}	
 	\pgfmathsetmacro\dU{0.15}
 	\pgfmathsetmacro\wU{0.2} 	
	
	\drawthreeblock{0}{$3$}{$p_1$}{$2$};
	\draw (3*\ovalx+2*\edge+\dU,0) node[text width=\wU] {$\bigcup$}; 
	\voval{(3*\ovalx+2*\edge+\ovaly+3*\dU+\wU,0)}{$p_2$};
	\draw ( \dtext,0) node [right] {if $p_2 >1$};

\end{tikzpicture}
& 
\bigstrut \\ \hline
 \end{tabulary}
 \caption{Log terminal primitive compactifications}
\label{t-table}
\end{table}
\end{center}

\begin{center}

\small
\begin{table}[htp]
\begin{tabulary}{\textwidth}{|b{3.05cm}|b{2.3cm}|b{4.8cm}|b{2.25cm}|}
%\toprule
\hline
$\vec \omega_e$ & Formal Newton pairs & Dual graph of resolution of singularities & $\gtwoa$-surface iff 
	\bigstrut \\ \hline
$(3p, 2p, 6p-6)$ \newline
 $p \geq 2$, $\gcd(p,6) = 1$
&$(2, 3), (- 6,p)$
& 
\begin{tikzpicture}[scale = 0.75]
	\pgfmathsetmacro\edge{0.5}
	\pgfmathsetmacro\vedge{.35}
	\pgfmathsetmacro\ovalx{.6}
	\pgfmathsetmacro\ovaly{0.25}
	\pgfmathsetmacro\vx{0.1}
	\pgfmathsetmacro\dtext{4}	
 	\pgfmathsetmacro\dU{0.15}
 	\pgfmathsetmacro\wU{0.2} 	
	
	\drawthreeblock{0}{$2$}{$3$}{$6$};
	\draw (3*\ovalx+2*\edge+\dU,0) node[text width=\wU] {$\bigcup$}; 
	\voval{(3*\ovalx+2*\edge+\ovaly+3*\dU+\wU,0)}{$p$};
\end{tikzpicture}
& 
$p \geq 3$ 
\bigstrut \\ \hline
 \end{tabulary}
 \caption{Primitive compactifications which are log canonical but not log terminal}
\label{c-table}
\end{table}
\end{center}

\appendix

\section{Proof of \cref{action-lemma}} \label{acsection}

At first we prove the $(\im)$ implication. The compatibility of the action implies that
\begin{subequations}
\begin{align}
a(t_1 + t_1',t_2 + t_2') &= a(t_1,t_2)a(t'_1,t'_2)\label{a-eqn} \\
b(t_1 + t_1',t_2 + t_2') &= b(t_1,t_2)b(t'_1,t'_2), \label{b-eqn} \\
c(t_1 + t_1',t_2 + t_2') &= b(t_1,t_2)c(t'_1,t'_2) + c(t_1,t_2), \label{c-eqn} \\
f(t_1 + t_1',t_2 + t_2',y) &= a(t_1,t_2)f(t'_1,t'_2,y) +  f(t_1,t_2,b(t'_1,t'_2)y + c(t'_1,t'_2)),\ \label{b-i-eqn} \text{where}\\
f(t_1,t_2,y) &:= \sum_{i=0}^m b_i(t_1,t_2)y^i
%\sum_{i=0}^m c_i(t_1 + t_1',t_2 + t_2')x^i &= c(t_1,t_2)\sum_{i=0}^m c_i(t'_1,t'_2)x^i +  \sum_{i=0}^m c_i(t_1,t_2)(a(t'_1,t'_2)x + b(t'_1,t'_2))^i \label{c-i-eqn}
\end{align}
\end{subequations}
Since $a,b$ are non-zero polynomials in $(t_1,t_2)$, identities \eqref{a-eqn} and \eqref{b-eqn} imply that $a(t_1,t_2) = b(t_1,t_2) = 1$ for all $(t_1,t_2) \in \GG^2_a$. Identity \eqref{c-eqn} then implies that $c$ is a linear function in $(t_1,t_2)$, i.e.\ $c(t_1,t_2) = c_1t_1 + c_2 t_2$ for some $c_1,c_2 \in \cc$. Consequently identity \eqref{b-i-eqn} implies that 
\begin{align}
\sum_{i=0}^m (b_i(t_1 + t_1',t_2 + t_2') - b_i(t'_1,t'_2))y^i = \sum_{i=0}^m b_i(t_1,t_2)(y+ c_1t'_1+ c_2 t'_2)^i \label{b-i-eqn'}
\end{align}
%If $\deg_y(f) \leq 0$ or 
If $c_1 = c_2 = 0$, then \eqref{b-i-eqn'} implies that each $b_i$ is linear and the action is given by 
\begin{align}
(t_1,t_2) \cdot (x,y) = \left(x + \sum_{i=0}^m b_i(t_1, t_2)y^i, y \right)
\end{align}
Now assume $(c_1, c_2) \neq (0,0)$. For all $\lambda \in \cc$, plugging in $(t'_1, t'_2) = (\lambda c_2, -\lambda c_1)$ in \eqref{b-i-eqn'} gives that 
\begin{align*}
\sum_{i=0}^m (b_i(t_1 +\lambda c_2, t_2 -\lambda c_1) - b_i(\lambda c_2, -\lambda c_1) - b_i(t_1,t_2))y^i  = 0
\end{align*}
so that 
\begin{align}
b_i(t_1 +\lambda c_2, t_2 -\lambda c_1) - b_i(\lambda c_2, -\lambda c_1) - b_i(t_1,t_2)  = 0 \label{b-i-eqn-2}
\end{align}
for each $i = 0, \ldots, m$. Let $\sigma: \cc^2 \to \cc^2$ be the map defined by 
\begin{align*}
(t_1, t_2) \mapsto (t_1, t_2) 
	\begin{pmatrix}
	\bar c_2 & c_1  \\
	-\bar c_1 & c_2
	\end{pmatrix}
\end{align*}  
where $\bar c_i$ denotes the conjugate ot $c_i$, $i = 1, 2$. Let $\tilde b_i := b_i \circ \sigma^{-1}$. Identity \eqref{b-i-eqn-2} implies that 
\begin{align}
\tilde b_i(s_1 +\lambda |c|^2, s_2) - \tilde b_i(\lambda |c|^2, 0) - \tilde b_i(s_1,s_2)  = 0,\quad i = 0, \ldots, m \label{b-i-eqn-3}
\end{align}
where we wrote $(s_1, s_2)$ for $\sigma(t_1, t_2)$ and $|c|^2$ for $|c_1|^2 + |c_2|^2$. It follows from \eqref{b-i-eqn-3} in a straightforward manner that for each $i = 0,\ldots, m$, 
\begin{align}
\tilde b_i(s_1,s_2) = g_i(s_2) + \mu_is_1
\end{align} 
for some $\mu_i \in \cc$ and $g_i \in \cc[s_2]$ such that $g_i(0) = 0$.  Plugging $b_i(t_1,t_2) = \tilde b_i \circ \sigma(t_1,t_2) = g_i(c_1t_1 + c_2t_2) + \mu_i(\bar c_2t_1 - \bar c_1 t_2)$ into \eqref{b-i-eqn'} gives 
\begin{align}
\sum_{i=0}^m (g_i(r+r') - g_i(r') + \mu_i s)x^i &= \sum_{i=0}^m ( g_i(r) + \mu_i s)(x+ r')^i \label{g-i-r-r'}
\end{align}
where $r := c_1t_1 + c_2t_2$, $r' :=c_1t_1 + c_2t_2$, and $s := \bar c_2t_1 - \bar c_1 t_2$, $i = 0, \ldots, s$. Substituting $r = 0$ in \eqref{g-i-r-r'} and using $g_i(0) = 0$ yields that
\begin{align}
\sum_{i=0}^m  \mu_i x^i &= \sum_{i=0}^m \mu_i (x+ r')^i \label{g-i-r-r'-1}
\end{align} 
which implies $\mu_i = 0$ for $i = 1, \ldots, m$. On the other hand, differentiating \eqref{g-i-r-r'} with respect to $r$
% then yields
%\begin{align}
%\sum_{i=0}^m g'_i(r+r')x^i = \sum_{i=0}^m g'_i(r)(x+ r')^i. \label{g'-eqn}
%\end{align}
%Substituting $r = 0$ into \eqref{g'-eqn} gives
and then substituting $r =0$ yields
\begin{align}
\sum_{i=0}^m g'_i(r')x^i = \sum_{i=0}^m g'_i(0)(x+ r')^i \label{g'-eqn-0}
\end{align}
A comparison of coefficients of $x^i$ from both sides of \eqref{g'-eqn-0} gives
\begin{align}
%g'_i(r') = \begin{cases}
%		g'_m(0) & \text{for}\ i = m,\\
%		g'_i(0) + g'_{i+1}(0)\binom{i+1}{1}r' + \cdots + g'_m(0)\binom{m}{m-i}r'^{m-i} & \text{for}\ 0, \ldots, m-1. .
%		 \end{cases} \label{g'-i}
g'_i(r') = 	g'_i(0) + g'_{i+1}(0)\binom{i+1}{1}r' + \cdots + g'_m(0)\binom{m}{m-i}r'^{m-i}, \quad i = 0, \ldots, m. \label{g'-i}
\end{align}
Since $g_i(0) = 0$ for each $i$, this implies that $g_i$'s are as in \eqref{g-i-1} with $\lambda_i := g'_i(0)$, $i = 0, \ldots, m$, and completes the proof of $(\im)$ direction. \\

Now we prove the $(\Leftarrow)$ implication. It suffices to show that if $(c_1, c_2) \neq (0,0)$, then identity \eqref{b-i-eqn'} holds with $b_i$'s defined in \eqref{b-i-1}. But then with $r, r'$ defined as in \eqref{g-i-r-r'}, identity \eqref{b-i-eqn'} is equivalent to the following identity: 
\begin{align*}
\sum_{i=0}^m (g_i(r+r') - g_i(r'))x^i 
	&=   \sum_{i=0}^m g_i(r) (x+ r')^i, 
\end{align*}
which is in turn equivalent to identities below:
\begin{align}
g_i(r+r') - g_i(r')
	&=   \sum_{j=i}^m  \binom{j}{j-i}(r')^{j-i}g_j(r),\ i = 0, \ldots, m. \label{g-i-diff}
\end{align}
Now \eqref{g-i-1} implies that for each $i = 0, \ldots, m$, 
\begin{align*}
g_i(r+r') - g_i(r')
	&=   \sum_{j=i}^m \frac{\lambda_j}{j-i+1}\binom{j}{j-i}((r+r')^{j-i+1} - r'^{j-i+1})\\
	&= \sum_{j=i}^m \frac{\lambda_j}{j-i+1}\binom{j}{j-i}\sum_{k=0}^{j-i} \binom{j-i+1}{k}r^{j-i+1-k}r'^k\\
	&= \sum_{k=0}^{m-i} r'^k\sum_{j=k+i}^m \frac{\lambda_j}{j-i+1}\binom{j}{j-i} \binom{j-i+1}{k}r^{j-i+1-k} \\
	&= \sum_{k=0}^{m-i} r'^k\sum_{j=k+i}^m \lambda_j\frac{j!}{i!k!(j-k-i+1)!}r^{j-k-i+1} \\
	&= \sum_{k=0}^{m-i}\frac{(k+i)!}{i!k!} r'^k\sum_{j=k+i}^m \frac{\lambda_j}{j-k-i+1}\frac{j!}{(k+i)!(j-k-i)!}r^{j-k-i+1} \\
	&= \sum_{k=0}^{m-i}\binom{k+i}{k} r'^k\sum_{j=k+i}^m \frac{\lambda_j}{j-k-i+1}\binom{j}{j-k-i}r^{j-k-i+1} \\
	&= \sum_{k=0}^{m-i}\binom{k+i}{k} r'^k g_{k+i}(r) \\
	&=   \sum_{j=i}^m  \binom{j}{j-i}(r')^{j-i}g_j(r),
\end{align*}
as required. 
\qed

\section{Automorphisms of minimal desingularizations of primitive compactifications of $\cc^2$}
In this section we show that every automorphism of a primitive compactification $\bar X$ of $X :=\cc^2$ lifts to an automorphism of the minimal desingularization $\bar X^{\min}$ of $\bar X$. Conversely, we also show that every automorphism of $\bar X^{\min}$ which fixes $X$ descends to an automorphism of $\bar X$.\\

Let $(Y,P)$ be a germ of a non-singular analytic surface. Choose analytical coordinates $(u,v)$ on $Y$ such that $P = \{u = v = 0\}$. Let $\tilde p,\tilde q$ be relatively prime positive integers such that $\tilde p > \tilde q \geq 1$, and let $\pi': Y' \to Y$ be the minimal resolution of the singularity of the curve $C :=  \{u(v^{\tilde p} - u^{\tilde q}) = 0\}$ at $P$ i.e.\
\begin{prooflist}
\item \label{resolution-1} $\pi'$ is an isomorphism outside the inverse image of $P$;
\item  \label{resolution-2} the strict transform of $u(v^{\tilde p} - u^{\tilde q}) = 0$ on $Y'$ intersects the union of the exceptional curves of $\pi'$ transversally;
\item  \label{resolution-3} every $\pi'': Y'' \to Y$ satisfying the above two properties factors through $\pi'$. 
\end{prooflist}
The morphism $\pi'$ can be expressed as a sequence of blow ups. Let $E_j$, $j = 1, 2, \ldots,$ be the strict transform of the $j$'th blow up on $Y'$. Denote by $E_0$ the strict transform of $u = 0$ on $Y'$. Given a germ $C$ of a curve at $P$, we say that $C$ is an {\em $E_j$-curvette} if the strict transform of $C$ on $Y'$ intersects $E_j$ transversally. The following lemma follows from standard theory of resolution of curve singularities. 

\begin{lemma} \label{c-fraction-lemma}
Express $\tilde p/\tilde q$ as a continued fraction in the following way:  
\begin{gather} \label{c-fractions}
\frac{\tilde p}{\tilde q} = m_1 + \cfrac{1}{m_2 +\cfrac{1}{\ddots + \cfrac{1}{m_N}}} 
\end{gather}
where $m_j \geq 2$, $j = 1, \ldots, N$. Then 
\begin{enumerate}
\item The dual graph of $E_0 \cup E_1 \cup \cdots$ is as in \cref{monomial-resolution}. 

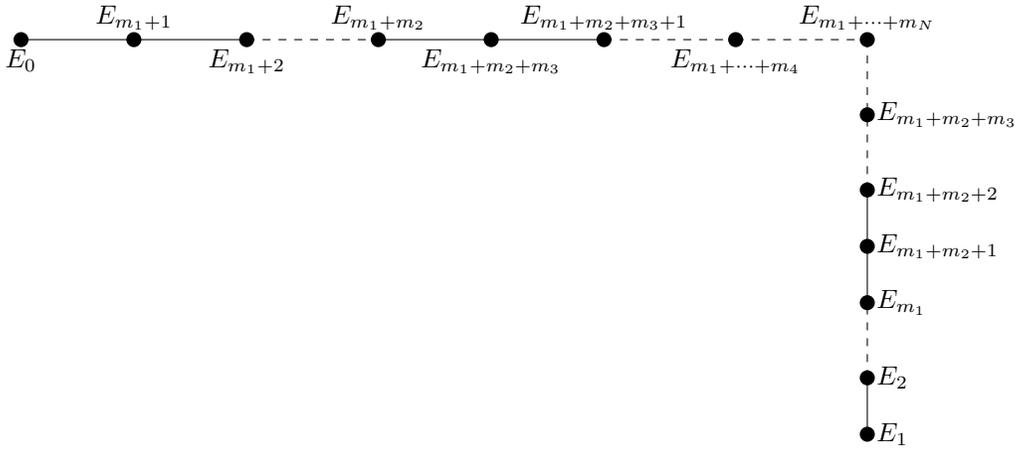
\begin{figure}[htp]

\newcommand{\hchain}[4]{
% 	\pgfmathsetmacro\x{0}
% 	\pgfmathsetmacro\y{0}

	\begin{scope}[shift={#1}]
		
		\draw (0 ,0) -- (\edge,0);
		\draw[dashed] ( \edge,0) -- (\edge + \dashededge,0);
		
		\fill[black] (0, 0) circle (\vr);
		\fill[black] (\edge, 0) circle (\vr);
		\fill[black] (\edge + \dashededge, 0) circle (\vr);
		
		\draw (0,0 )  node [above] {#2};
		\draw (\edge,0 )  node [above] {#3};
		\draw  (\edge + \dashededge, 0)   node  [above] {#4};
 	\end{scope}
 	}
 	
 	\newcommand{\hchaina}[4]{
 	
 		\begin{scope}[shift={#1}]
 			
 			\draw (0 ,0) -- (\edge,0);
 			\draw[dashed] ( \edge,0) -- (\edge + \dashededge,0);
 			
 			\fill[black] (0, 0) circle (\vr);
 			\fill[black] (\edge, 0) circle (\vr);
 			\fill[black] (\edge + \dashededge, 0) circle (\vr);
 			
 			\draw (0,0 )  node [above] {#2};
 			\draw (\edge,0 )  node [below] {#3};
 			\draw  (\edge + \dashededge, 0)   node  [above] {#4};
 	 	\end{scope}
 	 		
 	}
 	
 	\newcommand{\hchainb}[4]{
 	 	
 	 		\begin{scope}[shift={#1}]
 	 			
 	 			\draw (0 ,0) -- (\edge,0);
 	 			\draw[dashed] ( \edge,0) -- (\edge + \dashededge,0);
 	 			
 	 			\fill[black] (0, 0) circle (\vr);
 	 			\fill[black] (\edge, 0) circle (\vr);
 	 			\fill[black] (\edge + \dashededge, 0) circle (\vr);
 	 			
 	 			\draw (0,0 )  node [below] {#2};
 	 			\draw (\edge,0 )  node [above] {#3};
 	 			\draw  (\edge + \dashededge, 0)   node  [below] {#4};
 	 	 	\end{scope}
 	 	 		
 	 	}

\newcommand{\vchain}[4]{
% 	\pgfmathsetmacro\x{0}
% 	\pgfmathsetmacro\y{0}

	\begin{scope}[shift={#1}]
		
		\draw (0 ,0) -- (0,\vedge);
		\draw[dashed] (0, \vedge) -- (0, \vedge + \dashedvedge);
		
		\fill[black] (0, 0) circle (\vr);
		\fill[black] (0,\vedge) circle (\vr);
		\fill[black] (0, \vedge + \dashedvedge) circle (\vr);
		
		\draw (0,0 )  node [right] {#2};
		\draw (0,\vedge)  node [right] {#3};
		\draw (0, \vedge + \dashedvedge)  node  [right] {#4};
 	\end{scope}
 		
}

\begin{center}
\begin{tikzpicture}
 	\pgfmathsetmacro\edge{1.5}
  	\pgfmathsetmacro\dashededge{1.75}	
 	\pgfmathsetmacro\vedge{0.75}
 	\pgfmathsetmacro\dashedvedge{1}
	\pgfmathsetmacro\vr{0.1}
 	
 	\draw (0,0) -- (\edge,0);
 	\fill[black] (0, 0) circle (\vr);
 	\draw (0,0)  node [below] {$E_0$};
 	
 	\hchaina{(\edge,0)}{$E_{m_1+1}$}{$E_{m_1+2}$}{$E_{m_1+m_2}$}; 
	\draw (2*\edge + \dashededge,0) -- (3*\edge + \dashededge, 0);
	\hchainb{(3*\edge + \dashededge, 0)}{$E_{m_1+m_2+m_3}$}{$E_{m_1+m_2+m_3+1}$}{$E_{m_1+\cdots + m_4}$}; 
	\draw [dashed] (4*\edge +2* \dashededge,0) -- (4*\edge + 3*\dashededge, 0);	
 	\fill[black] (4*\edge + 3*\dashededge, 0) circle (\vr);
 	\draw (4*\edge + 3*\dashededge, 0)  node [above] {$E_{m_1 + \cdots + m_N}$}; 
 	
	\draw [dashed]  (4*\edge + 3*\dashededge, 0) -- (4*\edge +3* \dashededge,-\dashedvedge);
	\vchain{(4*\edge + 3*\dashededge, -2*\dashedvedge - \vedge )} {$E_{m_1+m_2+1}$}{$E_{m_1+m_2+2}$} {$E_{m_1+m_2 +m_3}$}; 
	\draw  (4*\edge + 3*\dashededge, -2*\dashedvedge - \vedge ) -- (4*\edge +3* \dashededge,  -2*\dashedvedge - 2*\vedge );
	\vchain{(4*\edge + 3*\dashededge, -3*\dashedvedge - 3*\vedge )} {$E_{1}$}{$E_{2}$} {$E_{m_1}$}; 			
\end{tikzpicture}
\caption{Dual graph for the minimal resolution of singularities of monomial curve singularities}\label{monomial-resolution}
\end{center}
\end{figure}
\item \label{E_0} The self intersection number of $E_0$ is $1- \lceil \tilde p/ \tilde q \rceil$. 
\item \label{curvette-assertion} Set $M_j  := \sum_{i=1}^j m_j$, $j = 0, \ldots, N$. For each $j = 0, \ldots, N-1$ and each $k = 1, \ldots, m_{j+1}$, let $\tilde p_{M_j + k}, \tilde q_{M_j + k}$ be the positive relatively prime integers such that 
\begin{gather} \label{c'-fractions}
\frac{\tilde p_{M_j+k}}{\tilde q_{M_j+k}} = m_1 + \cfrac{1}{m_2 +\cfrac{1}{\ddots + \cfrac{1}{m_j + \cfrac{1}{k}}}} 
\end{gather}
Then for generic $\xi' \in \cc$, the germ of $v^{\tilde p_{M_j+k}} - \xi' u^{\tilde q_{M_j+k}} = 0$ is an $E_{M_j+k}$-curvette. \qed
\end{enumerate}
\end{lemma}
%
%If $P$ is the 
%It is straightforward to check that for each $j = 0, \ldots, N-1$, 
%\begin{prooflist} [resume]
%\item  \label{odd} If $j$ is odd, then $\tilde p_{M_j+k}/\tilde q_{M_j+k} > \tilde p/\tilde q$ and $\lfloor \tilde q_{M_j+k}/\tilde p_{M_j+k} \rfloor = \lfloor \tilde q/\tilde p \rfloor = 0$.
%\item \label{even} If $j$ is even, then $\tilde q_{M_j+k}/\tilde p_{M_j+k} > \tilde q/\tilde p$ and $\lfloor \tilde p_{M_j+k}/\tilde q_{M_j+k} \rfloor = \lfloor \tilde p/\tilde q \rfloor = m_1$.
%\end{prooflist}

\begin{claim} \label{fractional-claim}
Adopt the notation of \cref{c-fraction-lemma}. Fix $j$, $0 \leq j \leq  N-1$.
\begin{enumerate}
\item \label{even} Assume $j$ is even. Then 
\begin{align}
(\tilde p_{M_j+k} - \tilde q_{M_j+k})/\tilde p_{M_j+k} 
	&< (\tilde p - \tilde q)/\tilde p \label{even-1} \\
 \lfloor (\tilde p_{M_j+k} - \tilde q_{M_j+k})/\tilde p_{M_j+k} \rfloor 
	 &= \lfloor (\tilde p - \tilde q)/\tilde p \rfloor \label{even-2}
	 = 0
\end{align}

\item  \label{odd} Assume $j$ is odd. Then 
\begin{align}
(\tilde p_{M_j+k} - \tilde q_{M_j+k})/\tilde p_{M_j+k} 
	&> (\tilde p - \tilde q)/\tilde p \label{odd-1}
\end{align}
Let $\Gamma$ be the weighted chain (where the weight of a vertex is the self intersection number of the corresponding curve) connecting $E_0$ to $E_{M_j+k}$. If $\Gamma$ is not as in \cref{fig:irrelevant}, then 
\begin{align}
 \lfloor \tilde p_{M_j+k} / (\tilde p_{M_j+k} - \tilde q_{M_j+k})\rfloor 
	 &\geq\lfloor \tilde p / (\tilde p - \tilde q) \rfloor\label{odd-2}
	% = m_2 + 1
\end{align}
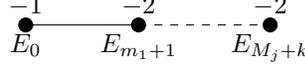
\begin{figure}[htp]
\newcommand{\hchain}[7]{
% 	\pgfmathsetmacro\x{0}
% 	\pgfmathsetmacro\y{0}

	\begin{scope}[shift={#1}]
		
		\draw (0 ,0) -- (\edge,0);
		\draw[dashed] ( \edge,0) -- (\edge + \dashededge,0);
		
		\fill[black] (0, 0) circle (\vr);
		\fill[black] (\edge, 0) circle (\vr);
		\fill[black] (\edge + \dashededge, 0) circle (\vr);
		
		\draw (0,0 )  node [below] {#2};
		\draw (\edge,0 )  node [below] {#4};
		\draw  (\edge + \dashededge, 0)   node  [below] {#6};
		
		\draw (0,0 )  node [above] {#3};
		\draw (\edge,0 )  node [above] {#5};
		\draw  (\edge + \dashededge, 0)   node  [above] {#7};
 	\end{scope}
 	}
\begin{center}
\begin{tikzpicture}
 	\pgfmathsetmacro\edge{1.5}
  	\pgfmathsetmacro\dashededge{1.75}	
 	\pgfmathsetmacro\vedge{0.75}
 	\pgfmathsetmacro\dashedvedge{1}
	\pgfmathsetmacro\vr{0.1}
 	
 	\hchain{(0,0)}{$E_0$}{$-1$}{$E_{m_1+1}$}{$-2$}{$E_{M_j+k}$}{$-2$}
 	 	
\end{tikzpicture}
\caption{`Irrelevant' weighted chain}\label{fig:irrelevant}
\end{center}
\end{figure}
\end{enumerate}
\end{claim}

\begin{proof}
Inequalities \eqref{even-1} and \eqref{odd-1} follows immediately from assertion \eqref{curvette-assertion} of \cref{c-fraction-lemma}. Since $\tilde q < \tilde p$, inequality \eqref{even-2} follows from \eqref{even-1}. We now prove \eqref{odd-2}. If $\tilde p/\tilde q > 2$ then
\begin{align*}
\lfloor \tilde p_{M_j+k}  /(\tilde p_{M_j+k} - \tilde q_{M_j+k}) \rfloor 
	\geq 1 
	 = \lfloor \tilde p / (\tilde p - \tilde q) \rfloor
\end{align*}
If $\tilde p /\tilde q = 2$, then the dual graph from \cref{monomial-resolution} is as follows, where we also list the weights (i.e.\ self intersection number of the corresponding curves):
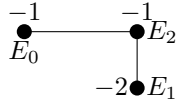
\begin{figure}[htp]
\begin{center}
\begin{tikzpicture}
 	\pgfmathsetmacro\edge{1.5}
  	\pgfmathsetmacro\dashededge{1.75}	
 	\pgfmathsetmacro\vedge{0.75}
 	\pgfmathsetmacro\dashedvedge{1}
	\pgfmathsetmacro\vr{0.1}
 	
 	\draw (0,0) -- (\edge,0);
 	\fill[black] (0, 0) circle (\vr);
 	\draw (0,0)  node [below] {$E_0$};
 	\draw (0,0)  node [above] {$-1$};
 	
 	\fill[black] (\edge, 0) circle (\vr);
 	\draw (\edge,0)  node [right] {$E_2$};
 	\draw (\edge,0)  node [above] {$-1$};
 	
 	\draw (\edge,0) -- (\edge,-\vedge);
	\fill[black] (\edge, -\vedge) circle (\vr);
 	\draw (\edge,-\vedge)  node [right] {$E_1$}; 	
 	\draw (\edge,-\vedge)  node [left] {$-2$};
 	 	
\end{tikzpicture}
\caption{Case $\tilde p/\tilde q = 2$}\label{fig:odd-2-1}
\end{center}
\end{figure}
Therefore \eqref{odd-2} is vacuously true. Now assume $\tilde p/\tilde q < 2$. Then $m_1 = 1$ and $N \geq 2$ in identity \eqref{c-fractions}. In particular, identities \eqref{c-fractions} and \eqref{c'-fractions} imply that 
\begin{align}
\lfloor \tilde p_{1+m_2}  /(\tilde p_{1+m_2} - \tilde q_{1+m_2}) \rfloor  
= \lfloor \tilde p/(\tilde p - \tilde q) \rfloor = 1 + m_2\label{<2}
\end{align}
It is straightforward to see that the weighted chain consisting of $E_0, E_2, E_3, \ldots, E_{m_2}$ is as in \cref{fig:irrelevant}, which proves \eqref{odd-2}. 
\end{proof}

Now we apply the preceding observations to minimal resolution of a primitive compactification $\bar X$ of $\cc^2$. Pick the (unique) primitive key sequence $\vec \omega = (\omega_0, \ldots, \omega_{n+1})$ in normal form and $\vec \theta \in \ntorus$ such that $\bar X \cong \xomegatheta$. As in \cref{primitive-dual-section} let $\tilde \phi(x,\xi) = \sum_{j=1}^s a_j x^{\beta_j} + \xi x^{\beta_{s+1}}$ be the formal descending Puiseux series associated to $\bar X$. Let $\beta_1 = \beta_{j_1} >  \cdots > \beta_{j_{l+1}} = \beta_{s+1}$ be the formal characteristic exponents, and $(q'_1, p_1), \ldots, (q'_{l+1}, p_{l+1})$ be the formal Newton pairs of $\tilde \phi$. Embed $X := \cc^2$ into $\pp^2$ via $(x,y) \mapsto [x:y:1]$. Then $(u,v) := (1/x,y/x)$ are analytic coordinates near $P := [1:0:0] \in \pp^2$; note that $u = 0$ is the equation of the line at infinity on $\pp^2$. Pick a generic $\xi' \in \cc$. Then
\begin{align*}
\tilde \psi(u,\xi') := u \tilde \phi(1/u, \xi)|_{\xi = \xi'} = \sum_{j=1}^s a_j u^{1- \beta_j} + \xi u^{1-\beta_{s+1}}
\end{align*}
is a (finite) Puiseux series in $u$. Let $C$ be the germ at $P$ of the (reduced) union of the line at infinity and the irreducible analytic curve with Puiseux expansion $v = \tilde \psi(u, \xi')$. It turns out (see e.g.\ \cite[Proposition 4.2]{sub2-1} that 
\begin{prooflist}[resume]
\item If  $\pi':\bar X' \to \pp^2$ is the minimal resolution (in the sense of properties \ref{resolution-1}-\ref{resolution-3}) of the singularity at $P$ of $C$, then $\bar X'$ is also a resolution of singularities of $\bar X$. 
\item The dual graph of the resolution $\sigma': \bar X' \to \bar X$ is of the form described in \cref{primitive-resolution-thm}. More precisely, in \cref{primitive-resolution}
\begin{prooflist}
\item the strict transform $E_0$ on $\bar X'$ of the line at infinity on $\pp^2$ is the `left end' of the leftmost chain (with $\Delta$-value $|q'_1| = q'_1$). 
\item \label{corner-observation} If $E$ corresponds to the vertex which is adjacent to both the chain with $\Delta$-value $|q'_i|$ and the chain with $\Delta$-value $p'_i$, $1 \leq i \leq l$, then for a generic $\xi' \in \cc$, the germ at $P$ of the curve with Puiseux expansion 
\begin{align}
v := \sum_{j < j_i} a_j u^{1- \beta_j} + \xi' u^{1-\beta_{j_i}} \label{corner-E}
\end{align}
is an $E$-curvette. 
\item \label{interior-observation} If $E$ corresponds to a vertex on the chain with $\Delta$-value $|q'_i|$ or $p'_i$, $1 \leq i \leq l$, then one of the following holds:
\begin{prooflist}
\item \label{interior-observation-1} there exists $j_*$, $j_{i-1} < j_* < j_i$, such that the germ at $P$ of the curve with Puiseux expansion 
\begin{align}
v := \sum_{j < j_*} a_j u^{1- \beta_j} + \xi' u^{1-\beta_{j_*}} \label{interior-E-1}
\end{align}
is an $E$-curvette for generic $\xi' \in \cc$; 
\item \label{interior-observation-2} or there exist $j_*$, $j_{i-1} < j_* \leq j_i$, and relatively prime positive integers $\tilde p_k, \tilde q_k$ which appear as exponents of curvettes from assertion \eqref{curvette-assertion} of \cref{c-fraction-lemma} with 
\begin{align}
(\tilde p, \tilde q) := 
	\begin{cases}
	(p_1, p_1 - q'_1) & \text{if}\ i = 1,\\
	(p_i, |q'_i|)	&\text{otherwise.}
	\end{cases}
\end{align}
such that the germ at $P$ of the curve with Puiseux expansion 
\begin{align}
v := \sum_{j < j_*} a_j u^{1- \beta_j} + \xi' u^{1-\beta_{j_{i-1}} + \tilde q_k/(p_1 \cdots p_{i-1}\tilde p_k)} \label{interior-E-2}
\end{align}
is an $E$-curvette for generic $\xi' \in \cc$.
\end{prooflist}
\end{prooflist}
\item \label{minimal-observation} The minimal resolution $\bar X^{\min}$ of singularities of $\bar X$ is formed by contracting some of the exceptional curves of $\pi'$, and possibly also the strict transform of the line at infinity. The latter gets contracted if and only if $q'_1 \leq p_1/2$, where $(q'_1, p_1), \ldots, (q'_{l+1}, p_{l+1})$ are formal Newton pairs of $\tilde \phi$. 
\end{prooflist}
Let $E$ be an exceptional curve of $\pi'$. Then $E$ defines a {\em divisorial valuation centered at infinity} on $\cc[x,y]$, and has an associated formal descending Puiseux series $\tilde \phi_E(x, \xi)$. Moreover, 
\begin{prooflist}[resume]
\item \label{descending-puiseux-observation} for each $\xi' \in \cc$, $u\tilde \phi_E(1/u,\xi)|_{\xi = \xi'}$ is precisely the Puiseux series from the right hand side of \eqref{corner-E}, \eqref{interior-E-1} or \eqref{interior-E-2} depending on the position of $E$. 
\end{prooflist}
Let $m_E$ be the integer associated to (the key sequence associated to) $E$ defined as in \eqref{m}, with $\vec \omega$ replaced by the key sequence associated to $E$. Observations \ref{corner-observation}, \ref{interior-observation} and \cite[Observation (i) from the proof of Theorem  5.2]{sub2-2} imply that 
\begin{align}
\momega
	&= \left\lfloor \frac{\ord_x(\tilde \phi)+1}{\deg_x(\tilde \phi)} - 1 \right\rfloor 
	= \left\lfloor\cfrac{p_1}{q'_1}\left (\beta_{s+1} + 1\right)  - 1 \right\rfloor  \label{momega} \\
m_E 
	&= \left\lfloor \frac{\ord_x(\tilde \phi_E)+1}{\deg_x(\tilde \phi_E)} - 1 \right\rfloor \label{m_E} \\
	&=
	\begin{cases}
	\left\lfloor\cfrac{\tilde p_k}{\tilde p_k - \tilde q_k} \right\rfloor  
		& \text{in the scenario of \ref{interior-observation-2} with $i = 1$,}  \\
	\left\lfloor\cfrac{p_1}{q'_1}\left (\beta_{j_{i-1}} - \cfrac{\tilde q_k}{p_1 \cdots p_{i-1}\tilde p_k} + 1\right)  - 1 \right\rfloor  
		& \text{in the scenario of \ref{interior-observation-2} with $i  > 1$,} \\
	\left\lfloor\cfrac{p_1}{q'_1}\left (\beta_{j_*} + 1\right)  - 1 \right\rfloor  
		& \text{in the scenario of \ref{interior-observation-1},} \\
	\left\lfloor\cfrac{p_1}{q'_1}\left (\beta_{j_i} + 1\right)  - 1 \right\rfloor  
		& \text{in the scenario of \ref{corner-observation}.} \\
	\end{cases}
	\notag
\end{align}

\begin{lemma} \label{m-lemma}
Let $E$ be an exceptional curve of the minimal resolution $\sigma: \bar X^{\min} \to \bar X$ of singularities of $\bar X$. Then $m_E \geq m_{\vec \omega}$, where $m_{\vec \omega}$ is as in \eqref{m}.  
\end{lemma}

\begin{proof}
Observation \ref{minimal-observation} implies that either 
\begin{defnlist}
\item \label{case-exceptional} $E$ comes from either an exceptional curve of $\pi': \bar X' \to \pp^2$,
\item \label{case-line} or $E$ is the strict transformation of the line at infinity on $\pp^2$.
\end{defnlist}
At first consider case \ref{case-exceptional}. In the scenario of observations \ref{corner-observation} or \ref{interior-observation-1} Identities \eqref{momega} and \eqref{m_E} immediately imply that $m_E \geq \momega$. Now note that 
\begin{align}
\momega \leq
 \left\lfloor\cfrac{p_1}{q'_1}\left (\beta_1+ 1\right)  - 1 \right\rfloor =  \left\lfloor\cfrac{p_1}{q'_1}\right\rfloor \label{momega-bound-1}
\end{align}
If \ref{interior-observation-2} holds with $i = 1$, then \eqref{even-1}, \eqref{odd-2} and \eqref{momega-bound-1} imply that $m_E \geq \momega$. On the other hand, if \ref{interior-observation-2} holds with $i > 1$, then 
%with $r := \lfloor |q'_i|/p_i \rfloor$, 
\eqref{even-2} and \eqref{odd-1} imply that 
\begin{align*}
\lfloor (\tilde p_k - \tilde q_k)/ \tilde p_k \rfloor
	&\geq \lfloor (p_i - |q'_i|)/ p_i \rfloor
\end{align*}
Since $p_1 \cdots p_{i-1}\beta_{j_{i-1}}$ is an integer, it follows that 
\begin{align*}
&	\lfloor (p_1 \cdots p_{i-1}\beta_{j_{i-1}}- \tilde q_k/ \tilde p_k \rfloor
		\geq \lfloor p_1 \cdots p_{i-1}\beta_{j_{i-1}} - |q'_i|/ p_i \rfloor
		= \lfloor p_1 \cdots p_{i-1}\beta_{j_i}\rfloor \\
\im		
&	\lfloor p_1 \cdots p_{i-1}(\beta_{j_{i-1}}- \tilde q_k/ (p_1 \cdots p_{i-1}\tilde p_k) + 1 ) \rfloor
		\geq  \lfloor p_1 \cdots p_{i-1}(\beta_{j_i} + 1)\rfloor \\
\im
& \lfloor (\beta_{j_{i-1}}- \tilde q_k/ (p_1 \cdots p_{i-1}\tilde p_k) + 1 )p_1/q'_1 \rfloor
		\geq  \lfloor (\beta_{j_i} + 1)p_1/q'_1\rfloor
		\geq \momega
\end{align*}
as required. \\

Now consider Case \ref{case-line}. Since $E_0$ does {\em not} get contracted, it follows from the arguments in the proof of \cref{fractional-claim} that $p_1/(p_1 - q'_1) > 2$. Identity \eqref{momega-bound-1} then implies that $\momega \leq 1 = m_E$. 
\end{proof}

Adopt the notation of \cref{m-lemma}. Let $\aut_X(\bar X)$ (resp.\ $\aut_X(\bar X^{\min} )$ be the set of automorphisms of $\bar X$ (resp.\ $\bar X^{\min} $) that fix $X$. 

\begin{thm} \label{min-aut}
%Let $\sigma: \bar X^{\min} \to \bar X$ be the minimal desingularization of a primitive normal compactification $\bar X$ of $X := \cc^2$.
\mbox{}
\begin{enumerate}
\item \label{lift} Every automorphism $F$ of $\bar X$ lifts to an automorphism of $\bar X^{\min}$ and $F(E) = E$ for every exceptional curve of $E$ of $\sigma$. 
\item \label{descend} Every automorphism of $\bar X^{\min}$ that fixes $X$ descends to an automorphism of $\bar X$.
\item \label{fixed-iso} $\sigma$ induces an isomorphism $\aut_X(\bar X ) \cong \aut_X(\bar X^{\min} )$. 
\end{enumerate}
%In particular, if $n \geq 1$ (where $\vec \omega = (\omega_0, \ldots, \omega_{n+1})$), then $\aut\bar X) \subseteq \aut(\bar X^{\min})$. 
\end{thm}

\begin{proof}
Let $F$ be an automorphism of $\bar X$. If $F(X) = X$, then \cref{m-lemma} and \cite[Theorem 4.9]{sub2-2} imply that assertion \eqref{lift} holds for $F$. 
%lifts to an automorphism of $\bar X^{\min}$ and $F(E) = E$ for every exceptional curve of $E$ of $\sigma$. 
%Recall that $\vec \omega = (\omega_0, \ldots, \omega_{n+1})$ is in the normal form. 
If $\bar X$ is not isomorphic to a weighted projective space of the form $\pp^2(1,1, p)$, then \cite[Proposition 5.1]{sub2-2} implies that every automorphism of $\bar X$ fixes $X$, so that assertion \eqref{lift} holds for $\bar X$. Now assume $\bar X \cong \pp^2(1,1,p)$. Since $\vec \omega$ is in the normal form, this implies (due to \cite[theorem 5.2]{sub2-2}) that $\vec \omega = (p,1)$. It is then straightforward to see that there is only one irreducible exceptional curve $E$ of $\sigma$ and the order of pole of a polynomial $f$ along $E$ is precisely $\deg_y(f)$. \cite[Theorems 4.9 and 5.2]{sub2-2}) then imply that assertion \eqref{lift} holds for $\bar X$. \\

Since $\vec \omega$ is in the normal form, observation \ref{descending-puiseux-observation} implies that for every irreducible exceptional curve of $E$ of $\sigma$, the `key sequence' of the pole along $E$ is in the `normal form' (in the sense of \cite[section 4]{sub2-2}) with respect to $(x,y)$-coordinates; moreover, the key sequences are distinct for distinct (irreducible) exceptional curve. \cite[Theorem 4.6]{sub2-2} then implies that $F(E) = E$ for every irreducible exceptional curve of $E$ of $\sigma$. Assertion \eqref{descend} then follows from \cite[theorems 4.9 and 5.2]{sub2-2}. Assertion \eqref{fixed-iso} is a consequence of the assertions \eqref{lift} and \eqref{descend}. 
\end{proof}

\bibliographystyle{alpha}
\bibliography{../../../utilities/bibi}
%\bibliography{bibi}

%% ***   NOTE   ***
%% If you don't use bibliography files, comment out the previous line
%% and use \begin{thebibliography}...\end{thebibliography}.  (In that
%% case, you should probably put the bibliography in a separate file
%% and `\include' or `\input' it here).

\end{document}

%% file: commands-2012.tex
\makeatletter

\newcommand{\Rmnum}[1]{\expandafter\@slowromancap\romannumeral #1@}
\makeatother

\DeclareMathOperator\ord{ord} 
 
\DeclareMathOperator\pole{pole}

\newcommand{\scrE}{\ensuremath{\mathcal{E}}}

\newcommand{\scrG}{\ensuremath{\mathcal{G}}}

\newcommand{\cc}{\ensuremath{\mathbb{C}}}

\newcommand{\GG}{\ensuremath{\mathbb{G}}}

\newcommand{\pp}{\ensuremath{\mathbb{P}}}
\newcommand{\qq}{\ensuremath{\mathbb{Q}}}

\newcommand{\zz}{\ensuremath{\mathbb{Z}}}

\newcommand{\torus}{\ensuremath{\mathbb{T}}}

\newcommand{\ntorus}{(\cc^*)^n}

\newcommand{\WP}{\ensuremath{{\bf W}\pp}}

\newcommand{\im}{\ensuremath{\Rightarrow}}

\newtheorem{thm}{Theorem}[section]
\newtheorem*{thm*}{Theorem}
\newtheorem{lemma}[thm]{Lemma}
\newtheorem*{lemma*}{Lemma}

\newtheorem{prop}[thm]{Proposition}
\newtheorem*{prop*}{Proposition}
\newtheorem{cor}[thm]{Corollary}
\newtheorem{claim}[thm]{Claim}
\newtheorem*{claim*}{Claim}

\newtheorem*{conjecture*}{Conjecture}

\theoremstyle{definition}

\newtheorem*{constrinition*}{Construction-Definition}

\newtheorem*{convention*}{Convention}
\newtheorem{defn}[thm]{Definition}
\newtheorem*{defn*}{Definition}

\newtheorem*{definotation*}{Definition-Notation}

\newtheorem*{example*}{Example}

\newtheorem*{fact*}{Fact}

\newtheorem*{facts*}{Facts}

\newtheorem*{bold-note*}{Note}
\newtheorem{problem}[thm]{Problem}
\newtheorem*{problem*}{Problem}
\newtheorem{bold-question}[thm]{Question}
\newtheorem*{bold-question*}{Question}
\newtheorem{rem}[thm]{Remark}

\newtheorem*{reminition*}{Remark-Definition}

\newtheorem{remexample*}{Remark-Example}

\newtheorem*{remtation*}{Remark-Notation}

\newtheorem*{remuestion*}{Remark-Question}

\newtheorem*{remvention*}{Remark-Convention}

\theoremstyle{remark}
\newtheorem*{rem*}{Remark}
\newtheorem*{note*}{Note}
\newtheorem*{notation*}{Notation}
\newtheorem*{question*}{Question}

\newtheorem*{questions*}{Questions}

\newcounter{UnorderedProofTempCtr}
\newcommand{\tempcommand}{}